\documentclass{amsart}
\usepackage[a4paper,outer=3.5cm,inner=3.5cm,total={14cm,21.4cm}]{geometry} 
\usepackage[english]{babel}
\usepackage[utf8]{inputenc}
\usepackage{amsmath}
\usepackage{amssymb}
\usepackage{amsthm}
\usepackage{enumerate}
\usepackage[shortlabels]{enumitem}
\usepackage{tikz-cd}
\usepackage{stmaryrd}
\usepackage{framed}

\usepackage[colorlinks=true]{hyperref}
\hypersetup{colorlinks, citecolor=blue, filecolor=black, linkcolor=black, urlcolor=black}

\usepackage[capitalize]{cleveref}
\crefformat{equation}{\normalfont(#2#1#3)}

\crefformat{section}{\normalfont \S#1}

\newtheorem{thm}{Theorem}[section]
\newtheorem{Theorem}[thm]{Theorem}
\newtheorem{Lemma}[thm]{Lemma}
\newtheorem{Proposition}[thm]{Proposition}
\newtheorem{Corollary}[thm]{Corollary}

\theoremstyle{definition}
\newtheorem{Definition}[thm]{Definition}
\newtheorem{Question}[thm]{Question}
\newtheorem{Example}[thm]{Example}
\newtheorem{Remark}[thm]{Remark}

\newlist{propenum}{enumerate}{1}
\setlist[propenum]{label=(\roman*), ref=\theProposition.(\roman*)}
\crefalias{propenumi}{Proposition}

\newlist{defenum}{enumerate}{1}
\setlist[defenum]{label=(\alph*), ref=\theDefinition.(\alph*)}
\crefalias{defenumi}{Definition}

\newcommand{\B}{\mathbb{B}}
\newcommand{\C}{\mathbb{C}}
\newcommand{\F}{\mathbb{F}}
\newcommand{\G}{\mathbb{G}}
\newcommand{\N}{\mathbb{N}}
\renewcommand{\O}{\mathcal{O}}
\renewcommand{\P}{\mathbb{P}}
\newcommand{\Q}{\mathbb{Q}}
\newcommand{\R}{\mathbb{R}}
\newcommand{\Z}{\mathbb{Z}}
\newcommand{\m}{\mathfrak{m}}

\newcommand{\uZp}{\underline{\Z}_p}

\newcommand{\whZ}{{\wh\Z}}
\newcommand{\uwhZ}{\underline{\wh\Z}}

\newcommand{\im}{\mathrm{im}}

\newcommand{\Hom}{\operatorname{Hom}}
\newcommand{\Map}{\operatorname{Map}}
\newcommand{\Mor}{\operatorname{Mor}}
\newcommand{\Mapc}{\operatorname{Map}_{\cts}}

\newcommand{\HOM}{\underline{\mathrm{Hom}}}
\newcommand{\EXT}{\underline{\mathrm{Ext}}}

\newcommand{\Perf}{\operatorname{Perf}}

\newcommand{\Spa}{\operatorname{Spa}}
\newcommand{\Spf}{\operatorname{Spf}}

\newcommand{\id}{{\operatorname{id}}}
\newcommand{\cts}{{\operatorname{cts}}}
\newcommand{\lc}{{\operatorname{lc}}}
\newcommand{\an}{\mathrm{an}}
\newcommand{\dR}{\mathrm{dR}}
\newcommand{\hotimes}{\hat{\otimes}}

\newcommand{\et}{{\mathrm{\acute{e}t}}}

\newcommand{\fet}{{\mathrm{f\acute{e}t}}}

\newcommand{\HT}{\mathrm{HT}}
\newcommand{\HTlog}{\mathrm{HTlog}}

\newcommand{\aeq}{\stackrel{a}{=}}
\newcommand{\ad}{\mathrm{ad}}
\newcommand{\wh}{\widehat}
\newcommand{\bOx}{\overline{\O}^{\times}}
\newcommand{\Ottt}{\G_m^{\tt}}
\newcommand{\bOttt}{\G_m/\G_m^{\tt}}
\newcommand{\Pic}{\operatorname{Pic}}
\newcommand{\uP}{\mathbf{Pic}}

\newcommand{\LSD}{\mathrm{LSD}}
\newcommand{\Spd}{\mathrm{Spd}}
\newcommand{\wt}{\widetilde}
\renewcommand{\diamond}{\diamondsuit}
\newcommand{\eq}{\mathrm{eq}}

\newcommand{\rk}{\mathrm{rk}}

\renewcommand{\lim}{\varprojlim}

\newcommand{\cH}{{\ifmmode \check{H}\else{\v{C}ech}\fi}}
\newcommand{\tf}{[\tfrac{1}{p}]}
\renewcommand{\tt}{\mathrm{tt}}
\newcommand{\tpt}{\langle p^\infty\rangle}
\newcommand{\uHiggs}{\mathbf{Higgs}}
\newcommand{\Bun}{\mathbf{Bun}}
\newcommand{\Higgs}{\mathbf{Higgs}}
\newcommand{\MB}{\mathbf M_{\mathrm{B}}}
\newcommand{\MD}{\mathbf M_{\mathrm{Dol}}}

\newcommand*\isomarrow{%
	\xrightarrow{\raisebox{-0.35em}{\smash{\ensuremath{\sim}}}}
}

\begin{document}
\title{A geometric $p$-adic Simpson correspondence in rank one}
\author{Ben Heuer}
\date{}
\maketitle
\begin{abstract}
	For any smooth proper rigid space $X$ over a complete algebraically closed extension $K$ of $\Q_p$ we give a geometrisation of the $p$-adic Simpson correspondence of rank one in terms of analytic moduli spaces: The $p$-adic character variety is canonically an \'etale twist of the moduli space of topological torsion Higgs line bundles over the Hitchin base. This also eliminates the choice of an exponential. The key idea is to relate both sides to moduli spaces of $v$-line bundles: We develop a theory of topological torsion subsheaves of $v$-sheaves and apply this to the diamantine $v$-Picard functor of \cite{heuer-diamantine-Picard}.
	
	\medskip
	
	As an application of this geometric correspondence, we study a major open question in $p$-adic non-abelian Hodge theory raised by Faltings, namely which Higgs bundles will correspond to continuous representations under the $p$-adic Simpson correspondence.  We answer this question in rank one by describing the essential image of the continuous characters $\pi^{\et}_1(X)\to K^\times$ in terms of moduli spaces: For projective $X$ over $K=\C_p$, it is given by Higgs line bundles with vanishing Chern classes like in complex geometry, but in general we show that the correct condition is the strictly stronger assumption that the underlying line bundle is a topological torsion element in the topological group $\Pic(X)$.
	
\end{abstract}

\setcounter{tocdepth}{2}
\section{Introduction}

The Corlette--Simpson correspondence for a smooth projective variety $X$ over $\C$ is an equivalence of categories between finite dimensional $\C$-linear representations of the topological fundamental group $\pi_1(X)$ of $X(\C)$ on the one hand, and semi-stable Higgs bundles on $X$  with vanishing (rational) Chern classes on the other hand \cite[\S1]{SimpsonCorrespondence}. For any $n\in \N$, there are natural complex analytic moduli spaces of the rank $n$ objects on either side, and the correspondence induces a homeomorphism between them, see  \cite[Proposition~1.5]{SimpsonCorrespondence}\cite[Theorem 7.18]{SimpsonModuliII}. However, this map does not respect the complex analytic structures.

Our first result is a very close $p$-adic analogue of this equivalence in rank one:

\begin{Theorem}\label{t:intro-Corollary-Cp}
	Let $X$ be a connected smooth projective variety over $\C_p$ and fix $x\in X(\C_p)$. Then any choice of an exponential on $\C_p$ and of a $B_{\mathrm{dR}}^+/\xi^2$-lift of $X$ induce an equivalence
	\[\Big\{\begin{array}{@{}c@{}l}\text{1-dimensional continuous}\\\text{$\C_p$-representations of $\pi_1^{\et}(X,x)$} \end{array}\Big\} \isomarrow \Big\{\begin{array}{@{}c@{}l}\text{Higgs bundles on $X$ of rank $1$}\\\text{with vanishing Chern classes}\end{array}\Big\}.\]
	Upon passing to isomorphism classes, this induces an isomorphism of topological groups 
	\[ \Hom_{\cts}(\pi_1^{\et}(X,x),\C_p^\times)\cong \Pic^{\tau}(X)\times H^0(X,\Omega^1(-1)).\]
	This does not come from a rigid analytic map for the natural rigid space structures, see \S1.1.
	
	\medskip
	
	More generally, let $X$ be a connected smooth \textnormal{proper} rigid space over \textnormal{any} complete algebraically closed field $K|\C_p$ and assume that the rigid analytic Picard functor of $X$ is representable by a rigid space $\uP_X$. Then  the analogous statements hold when we replace the condition of ``vanishing Chern classes'' by the strictly stronger condition that the underlying line bundle $L$ is ``topologically torsion'' in  $\uP_X(K)$, meaning that $L^{n!}\to 1$ for $n\to \infty$.
\end{Theorem}
This is the first instance of a \textit{correspondence} relating $p$-adic continuous representations to Higgs bundles with an explicit topological condition. It gives useful insights into what to expect in higher rank, e.g.\ that Chern classes do not define the correct condition in general.

\subsection{Geometrising the $p$-adic Simpson correspondence of rank one}

The main aim of this article is to show that, surprisingly, \cref{t:intro-Corollary-Cp} can be upgraded to a \textit{geometric} statement about moduli spaces: 
We will show that there is always a smooth rigid analytic moduli space $\uP^\tt_{X,\et}$ of topological torsion line bundles. This is an open rigid subgroup of the Picard variety $\uP_{X}$ if the latter exists, but in fact we do not need representability of the Picard functor for the construction.  Using the Hitchin base  $\mathcal A$ of rank one,  this allows us to define the ``Dolbeault moduli space'' of topologically torsion Higgs line bundles
\[\MD:=\uP_{X,\et}^{\tt}\times \mathcal A,\quad \text{where } \mathcal A:=H^0(X,\Omega^1(-1))\otimes_K\G_a.\]
On the other side of the $p$-adic Simpson correspondence, we have the $p$-adic character variety
\[\MB:=\HOM(\pi_1^\et(X,x),\G_m)\]
defined as the internal Hom in $v$-sheaves over $K$, where the profinite group $\pi_1^\et(X,x)$ is considered as a profinite $v$-sheaf. As we explain in \cref{s:character-variety}, the functor $\MB$ is represented by a rigid group variety, playing the role of the ``Betti moduli space'' in this context. We show:

\begin{Theorem}[\cref{t:geometric-Simpson}]\label{t:padic-Simpson-intro}
	Let $X$ be a connected smooth proper rigid space over $K$ and fix $x\in X(K)$. Then there is a natural short exact sequence of rigid analytic groups
	\begin{equation}\label{intro:seq-over-K}
		0\to \uP_{X,\et}^{\tt}\to \MB\xrightarrow{\HTlog} \mathcal A\to 0.
	\end{equation}
 On tangent spaces, the associated sequence of Lie algebras recovers the Hodge--Tate sequence
	\begin{equation}\label{r:HT-seq-splitting}
	0\to H^1_{\an}(X,\O)\to H^1_{\et}(X,\Q_p)\otimes_{\Q_p} K\xrightarrow{\HT} H^0(X,\Omega^1(-1))\to 0.
	\end{equation}
\end{Theorem}
Here exactness of \eqref{intro:seq-over-K} is to be understood with respect to the \'etale topology. While the map $\HTlog$ in \eqref{intro:seq-over-K} can be defined explicitly using the $p$-adic logarithm, the first map is more subtle: It generalises and geometrises a construction of Deninger--Werner \cite{DeningerWerner_vb_p-adic_curves}, as well as of Song \cite{song2020rigid} in the case of curves over $\overline{\Q}_p$
as we explain in \S4. It moreover induces a generalised Weil pairing
\[ \uP_{X,\et}^{\tt}\times \pi_1^{\et}(X,x)\to \G_m.\]

We use \cref{t:padic-Simpson-intro} to describe a geometric $p$-adic Simpson correspondence in rank one in terms of a comparison of moduli spaces: The projection $\MD\to \mathcal A$ may be interpreted as the Hitchin fibration. 
Theorem~\ref{t:padic-Simpson-intro} now yields an analogous map $\HTlog$ on the Betti side:
\[ \begin{tikzcd}[row sep = 0cm,column sep = 1cm]
	\MB\arrow[rd,"\HTlog"'] & &	\MD \arrow[ld,"\mathrm{Hitchin}"] \\
	&  \mathcal A
\end{tikzcd}\]
From this perspective, Theorem~\ref{t:padic-Simpson-intro} says that both $\MB$ and $\MD$ are $\uP_{X,\et}^{\tt}$-torsors over $\mathcal A$, but the latter is split, while we show that the former is never split outside trivial of cases.

It follows from this that we can regard $\MB $ as an \'etale twist of $\MD$. In fact, our main result is that one can more canonically compare these two torsors in a geometric fashion:
\begin{Theorem}[\cref{t:etale-comparison}]\label{t:intro-comparison-of-moduli-spaces}
	Any choice of a $B^+_{\mathrm{dR}}/\xi^2$-lift $\mathbb X$ of $X$ induces a natural $\uP_{X,\et}[p^\infty]$-torsor $\mathbb L_{\mathbb X}\to\mathcal A$ for which there is a canonical isomorphism of rigid spaces
	\[\MB\times_{\mathcal A}\mathbb L_{\mathbb X}\isomarrow \MD\times_{\mathcal A}\mathbb L_{\mathbb X}.\]
\end{Theorem}
Here the local system $\mathbb L_{\mathbb X}$ can roughly be thought of as a moduli space of exponentials.
We also give a completely canonical variant of \cref{t:intro-comparison-of-moduli-spaces} over a moduli space of all lifts. These results explain the choices necessary in Faltings' $p$-adic Simpson correspondence in a geometric fashion: Any choice of an exponential induces a splitting of $\mathbb L_{\mathbb X}\to\mathcal A$ on $K$-points.
We take this as a first sign that in the $p$-adic situation, a more geometric formulation of the $p$-adic Simpson correspondence is possible. We will pursue this further  in \cite{heuer-sheafified-paCS}.
\begin{Remark}\label{r:splitting-of-HT}
	A choice of $\mathbb X$ is equivalent to a choice of splitting of the Hodge--Tate sequence \cref{r:HT-seq-splitting},
	see \cite[Proposition 7.2.5]{Guo_HodgeTate}. If we are given a model of $X$ over a local field $L|\Q_p$, this induces such a lift $\mathbb X$, and Theorem~\ref{t:intro-comparison-of-moduli-spaces} then provides a  completely canonical comparison.
\end{Remark}

\begin{Remark}\label{r:Remark-comparison-to-complex}
	It is very surprising to us that the $p$-adic Simpson correspondence in rank one allows for a more refined comparison of the rigid analytic structures as in Theorem~\ref{t:intro-comparison-of-moduli-spaces}: If $X$ is instead a smooth projective curve over $\C$, the analogue of \eqref{intro:seq-over-K} is the exact sequence
	\begin{equation}\label{intro:seq-over-C}
		0\to \Pic^0(X)\to \Hom(\pi_1(X),\C^\times)\to H^0(X,\Omega^1)\to 0
	\end{equation}
	where the first map is defined as
	$\Pic^0(X)=\Hom(\pi_1(X),S^1)\to \Hom(H_1(X,\Z),\C^\times)$ using the decomposition $\C^\times=S^1\times \R^\times_+$,
	see \cite[p21]{SimpsonCorrespondence}\cite[\S2]{ModuliRankOne}.
	This is clearly only real-analytic and not complex-analytic.
	Instead, the only result we know in the literature that resembles Theorem~\ref{t:intro-comparison-of-moduli-spaces} is Groechenig's result  in the mod $p$-Simpson correspondence that for curves over $\overline{\F}_p$, the de Rham moduli stack is a twist of the Dolbeault moduli stack \cite{Groechenig_modp-Simpson}.
\end{Remark}

We now discuss the precise relation to previous results in the $p$-adic Simpson correspondence, and explain in more detail the topological torsion condition and the moduli spaces.

\subsection{The conjectural $p$-adic Simpson correspondence: known results}
Let $K|\Q_p$ be any complete algebraically closed extension
 and let $X$ be a connected smooth proper rigid space over $K$ with fixed base-point $x\in X(K)$. In this setting, the conjectural $p$-adic Simpson correspondence is expected to be an equivalence of categories (depending on choices)
 \begin{equation}\label{eq:conjectural-Simpson}
	\Big\{\begin{array}{@{}c@{}l}\text{finite dimensional continuous }\\\text{$K$-linear representations of $\pi_1^{\et}(X,x)$} \end{array}\Big\} \isomarrow \Big\{\begin{array}{@{}c@{}l}\text{Higgs bundles on $X$}\\\text{satisfying...??}\end{array}\Big\},
\end{equation}	
	where ``??'' is a condition yet to be identified. 
In order to construct a fully faithful functor from representations to Higgs bundles, Faltings \cite{Faltings_SimpsonI} introduced a category of ``generalised representations'' into which representations of $\pi_1^{\et}(X,x)$ embed fully faithfully. This category can be shown to be equivalent to the category of $v$-vector bundles on $X$ \cite[Proposition~2.3]{heuer-G-torsors-perfectoid-spaces}. If we gloss over some technical differences in setups, then reformulated in this language, Faltings proved that if $X$ is a smooth projective curve, there is an equivalence
\begin{equation}\label{eq:conj-non-ab-Hodge-corresp}
\{\text{$v$-vector bundles on $X$}\} \isomarrow \{\text{Higgs bundles on $X$}\}.
\end{equation}
Such an equivalence is expected to exist for any smooth proper rigid space $X$ over $K$. While this is currently not known in general,
Faltings constructs a ``local'' version of this correspondence  under additional ``smallness'' assumptions, which was reinterpreted in terms of period rings and studied in detail by Abbes--Gros and Tsuji \cite{AGT-p-adic-Simpson}. There are other known instances, e.g.\ due to Wang for $X$ of good reduction \cite{Wang-Simpson}. In rank one, we constructed an equivalence for general $X$ by way of a ``Hodge--Tate sequence for $\G_m$'' for the group $\Pic_v(X)$ of isomorphism classes of $v$-line bundles on $X$:  a short exact sequence
\begin{equation}\label{eq:pic-HT-seq}
0\to \Pic_{\et}(X)\to \Pic_v(X)\to  H^0(X,\Omega^1(-1))\to 0,
\end{equation}
see \cite{heuer-v_lb_rigid}.
We note that neither of the (conjectural) equivalences \cref{eq:conjectural-Simpson} and \cref{eq:conj-non-ab-Hodge-corresp} are expected to be canonical, instead they should depend on choices of an exponential for $K$ and a $B_{\dR}^+/\xi^2$-lift of $X$. For example, these choices induce a splitting of \cref{eq:pic-HT-seq}, which is not canonically split.

\medskip

Given correspondence \cref{eq:conj-non-ab-Hodge-corresp} in any setup where it is known,
a fundamental open problem in $p$-adic non-abelian Hodge theory, raised by Faltings in \cite[\S5]{Faltings_SimpsonI}, is to find condition ``??'':
\begin{Question}\label{q:condition-Q}
 Which Higgs bundles on $X$ correspond to continuous representations of $\pi_1^{\et}(X,x)$ under the (in general conjectural) $p$-adic non-abelian Hodge correspondence \cref{eq:conj-non-ab-Hodge-corresp}?
\end{Question}

In the case of vanishing Higgs field, this question had been studied by Deninger--Werner \cite{DeningerWerner_vb_p-adic_curves}\cite{DeningerWerner_Simpson} prior to Faltings' article. Based on their work, Xu has recently made progress on \cref{q:condition-Q} in the case of curves by identifying a condition for Higgs bundles to be ``potentially Deninger--Werner'', from which he is able to construct a functor to representations \cite{XuTransport_parallele}. This raises the question if this condition has a more classical description that can be used in practice to see whether a Higgs bundle  is ``potentially Deninger--Werner''.

The construction of Deninger--Werner has been extended to the rigid analytic setup by W\"urthen \cite[\S3]{wuerthen_vb_on_rigid_var}, whose results suggest that in the special case of vanishing Higgs field, pro-finite-\'etale vector bundles give the correct subcategory of Higgs bundles. Here a vector bundle is called pro-finite-\'etale if it becomes trivial on a pro-finite-\'etale cover 
in Scholze's  pro-\'etale site \cite{Scholze_p-adicHodgeForRigid} (we note that this is strictly stronger than being trivialised by a finite \'etale cover). Such vector bundles were further studied in \cite{MannWerner_LocSys_p-adVB}.
We showed in \cite[\S5]{heuer-v_lb_rigid} that the pro-finite-\'etale condition also gives the correct category for all Higgs bundles in rank one, namely we constructed an equivalence as in \cref{eq:conjectural-Simpson} where the left hand side consist of continuous characters and the right hand side of Higgs line bundle whose underlying line bundle is pro-finite-\'etale. In \cite{HMW-abeloid-Simpson}, we deduced that on abelian varieties, the answer to \cref{q:condition-Q} is given by pro-finite-\'etale Higgs bundles also in higher rank.

However, these results in turn raise the question how pro-finite-\'etale vector bundles can be characterised more classically, as it is a priori difficult to verify that a given vector bundle is pro-finite-\'etale. In \cite[\S5]
{heuer-v_lb_rigid}, we characterised pro-finite-\'etale line bundles  as those admitting a reduction of structure group to a certain subsheaf $\G_m^{\tt}\subseteq \G_m$, but we left it open how such line bundles can be described more concretely, e.g.\ in terms of the Picard variety.

\medskip

 \cref{t:intro-Corollary-Cp} now gives a fully satisfactory answer to this question, thus to \cref{q:condition-Q}, for line bundles on any smooth proper rigid space. One reason why we think this description is the ``correct'' one is that it is geometric, i.e.\ it can be phrased in terms of moduli spaces, opening up a completely new possibility for a geometric formulation of the $p$-adic Simpson correspondence: We believe that \cref{t:intro-comparison-of-moduli-spaces} stands a chance to generalise to higher rank, as we will explore further in \cite{heuer-sheafified-paCS}. Moreover, this suggests that a moduli-theoretic approach to non-abelian Hodge theory may help answer \cref{q:condition-Q} in general.

\subsection{A new geometric approach via diamantine Picard functors}
We now elaborate on our strategy to prove \cref{t:intro-Corollary-Cp}, \cref{t:padic-Simpson-intro} and \cref{t:intro-comparison-of-moduli-spaces}, as well as on the construction of the moduli spaces, which yields additional results of independent interest:

A crucial new technical ingredient to this article are the diamantine Picard functors introduced for this purpose in \cite{heuer-diamantine-Picard}:
To motivate this, let us first mention that we will show that a line bundle $L$ on $X$ is pro-finite-\'etale if and only if it extends to a line bundle on the adic space $X\times \uwhZ$ whose specialisation at each $n\in \Z\subseteq \wh{\Z}$ is isomorphic to $L^n$. Here $\uwhZ$ denotes the adic space over $\Spa(K)$ which represents the $v$-sheaf $\varprojlim_{N\in \N} \underline{\Z/N\Z}$ over $K$.

We would like to say that this induces a morphism of adic spaces from $\uwhZ$ to the rigid Picard variety of $X$. However, $\uwhZ$ is no longer a rigid space, rather it is perfectoid. We therefore use the diamantine Picard functor $\uP_{X}$ defined on perfectoid test objects. More precisely, according to \cref{eq:pic-HT-seq}, there are two different such functors, one for \'etale line bundles, denoted by $\uP_{X,\et}$, the other for $v$-line bundles, denoted by $\uP_{X,v}$. Explicitly, for $\tau \in \{\et,v\}$,
\[ \uP_{X,\tau}:\Perf_K\to \mathrm{Ab}, \quad T\mapsto \Pic_\tau(X\times T)/\Pic_\tau(T).\]

In order to characterise maps to $\uP_{X,\tau}$ that ``extend $\whZ$-linearly'' as above, we now define:

\begin{Definition}
	For any $v$-sheaf $F$ on $\Perf_K$, the topological torsion subsheaf $F^\tt\subseteq F$ is the sheaf-theoretic image of the map
	$\HOM(\uwhZ,F)\to F$ given by evaluation at $1\in \whZ$.
\end{Definition}
We will give a more classical description of $F^\tt$ when $F$ is a rigid group:
For example, the aforementioned topological torsion subsheaf  $\G_m^\tt\subseteq \G_m$ can be shown to be represented by the open subgroup generated by the open disc around $1$ and the roots of unity.

These definitions, combined with a systematic study of topological torsion subsheaves, allow us to give a much more conceptual characterisation of pro-finite-\'etale line bundles than the one we previously gave in \cite[Theorem~5.7.3]{heuer-v_lb_rigid}: Namely, in the category of diamonds, there is a universal pro-finite-\'etale cover
$\wt X\to X$,
a pro-finite-\'etale torsor under $\pi^{\et}_1(X,x)$. The following result now describes which line bundles are trivialised by pro-finite-\'etale covers, or equivalently by $\wt X$, and thus give rise to $p$-adic characters:

\begin{Theorem}[\cref{t:tt-torsion-is-1+m-torsors}]\label{t:Pictt-is-profet-intro}
	Let $\tau\in\{\et,v\}$.
	There is a short exact sequence of sheaves 
	\[ 0\to \uP^{\tt}_{X,\tau}\to \uP_{X,\tau}\to \uP_{\wt X,\tau}\]
	on $\Perf_{K,\tau}$.
The Cartan--Leray sequence of $\wt X\to X$ thus induces a canonical isomorphism
\[\uP^{\tt}_{X,v}=\HOM(\pi_1^{\et}(X,x),\G_m).\]
\end{Theorem}
Here we crucially use that $\uP_{X,\tau}$ is defined on perfectoid test objects such as $\uwhZ$. On $K$-points, we deduce that the line bundles on $X$ trivialised by $\wt X$ are given by
$\uP^{\tt}_{X,\et}(K)=\Pic(X)^\tt$,
i.e.\ those classes $L$ in the topological group $\Pic(X)$ for which $L^{n!}\to 1$ for $n\to \infty$.
To deduce Theorem~\ref{t:intro-Corollary-Cp}, we show that for projective $X$ over $\C_p$, we in fact have
\begin{equation}\label{eq:Pic-tt-vs-Pictau}
\uP_{X,\et}^{\tt}(\C_p)=\Pic_{\et}^{\tau}(X)
\end{equation}
where $\Pic_{\et}^{\tau}(X)$ is the group of line bundles with torsion N\'eron--Severi class.
These are precisely the line bundles with vanishing (rational) Chern classes. To see \eqref{eq:Pic-tt-vs-Pictau}, we use that $\uP^0_{X,\et}$ is represented by an abelian variety $A$ over $\C_p$, and we show that  $A^\tt(\C_p)=A(\C_p)$. For this we use that the residue field of $\C_p$ is a union of finite fields, and that $v_p(\C_p)=\Q$. 

On the other hand, this shows that finding the correct condition ``??'' in \eqref{eq:conjectural-Simpson} is very subtle:
\begin{Remark}
	Over any extension $K\supsetneq\C_p$,  the  inclusion $\uP_{X,\et}^{\tt}(K)\subseteq \Pic^{\tau}(X)$ can become strict: If $X$ is a curve of good reduction over $K$ with Jacobian $J=\uP_{X,\et}^{0}$, then  $J(K)^\tt$ are precisely those points of $J(K)$ that reduce to a torsion point on the special fibre. This special case of curves with good reduction is implicit in Faltings' discussion in \cite[p.856]{Faltings_SimpsonI}.
\end{Remark}

\begin{Remark}
	Already over $\C_p$, the equality \eqref{eq:Pic-tt-vs-Pictau} and thus the first part of \cref{t:intro-Corollary-Cp} can become false if we only assume that $X$ is proper rather than projective, since  $\uP^0_X$ is then no longer necessarily an abelian variety: For example, if $X$ is a Hopf variety, then $\uP_X=\G_m$ (see \cite[Remark~0.1.3]{HartlLutk}).  But we have $\G_m^\tt(\C_p)=\O_{\C_p}^\times\subsetneq \C_p^\times=\G_m(\C_p)$.
\end{Remark}
\begin{Remark} In terms of rigid spaces in the sense of Tate, the map $\uP_{X,\et}^{\tt}\hookrightarrow \uP_{X,\et}$ may by \cref{eq:Pic-tt-vs-Pictau} be an open bijective map that is no longer surjective when passing to adic spaces.
\end{Remark}

From this perspective,  it arguably feels like a coincidence that ``vanishing Chern classes'' are the correct answer to \cref{q:condition-Q} for line bundles on projective varieties over $\C_p$.
Instead, as already in rank one the cleanest formulation of the condition of ``topological torsion'' requires the geometric structure of $\uP_{X,\et}$ as a rigid space, this suggests that in general, analytic moduli spaces may be necessary to answer \cref{q:condition-Q} over general fields.

\subsection*{Structure of the article} In section 2 we introduce the notion of topological torsion subsheaves and study these systematically. In particular, with an eye towards applications to Picard varieties, we study topological torsion subsheaves of rigid groups.

In section 3 we introduce the $p$-adic character variety and explain why this is isomorphic to the topological torsion subsheaf of the $v$-Picard functor, that is we prove \cref{t:Pictt-is-profet-intro}. 

In section 4 we prove \cref{t:padic-Simpson-intro}. We also give the application to Deninger-Werner's map by constructing a generalised Weil pairing. As a further application, we comment on  applications to the study of the structure of rigid Picard functors.

In section 5, we prove the \'etale comparison of Betti and Dolbeault moduli spaces, \cref{t:intro-comparison-of-moduli-spaces}, as well as several related comparison isomorphisms.

\subsection*{Acknowledgements}
We would like to thank Johannes Ansch\"utz, Yajnaseni Dutta,  Arthur-C\'esar Le Bras, Lucas Mann, Peter Scholze, Alexander Thomas, Annette Werner and Matti W\"urthen for very helpful discussions.

This work was funded by the Deutsche Forschungsgemeinschaft (DFG, German Research Foundation) under Germany's Excellence Strategy-- EXC-2047/1 -- 390685813. The author was supported by the DFG via the Leibniz-Preis of Peter Scholze.
\subsection*{Notation}
Throughout, let $K$ be a (complete) non-archimedean field of residue characteristic $p$. We denote by $\O_K=K^{\circ}$ the ring of integers consisting of power-bounded elements. More generally, we denote by $K^+\subseteq K$ any ring of integral elements.
Let $\m\subseteq K^+$ be the maximal ideal, and $\Gamma$ the value group. From \S3 on, $K$ will be a perfectoid field extension of $\Q_p$ in the sense of \cite{perfectoid-spaces}, and we then use almost mathematics with respect to $\m$.

By a rigid space over $(K,K^+)$ we mean an adic space locally of topologically finite type over $\Spa(K,K^+)$.  In the following, we shall often abbreviate $\Spa(K,K^+)$ by $\Spa(K)$ when the ring $K^+$ is clear from the context. In particular, we then denote by $\mathrm{SmRig}_{K}$ the category of smooth rigid spaces over $\Spa(K,K^+)$, by $\Perf_{K}$ the category of perfectoid spaces over $(K,K^+)$, and by $\mathrm{LSD}_{K}$ the category of locally spatial diamonds over $\Spd(K,K^+)$ in the sense of \cite[\S11]{etale-cohomology-of-diamonds}. More generally, for any rigid space $X$ we denote by $\mathrm{SmRig}_{X}$ the slice category of rigid spaces over $X$.
In accordance with this notation, for any two rigid spaces $X$ and $Y$ over $K$, we also write $Y_X:=Y\times_KX$. Similarly for the other categories.

We denote by $\B^d$ the rigid unit polydisc of dimension $d$ over $(K,K^+)$, by $\G_m$ the rigid torus over $(K,K^+)$ defined as the analytification of the algebraic torus, and similarly by $\G_a$ the rigid affine line with its additive structure. We  denote by $\wh{\G}_m$ the open unit disc around $1$ in $\G_m$. These are all examples of rigid groups: 
By a rigid group we mean a group object in the category of rigid spaces over $(K,K^+)$. In this article, we always assume this to be commutative without further mention. Similarly for adic groups.

There is a fully faithful diamondification functor $\mathrm{SmRig}_{K}\hookrightarrow \mathrm{LSD}_{K}$, $X\mapsto X^\diamondsuit$, see \cite[\S15.6]{etale-cohomology-of-diamonds}, and we shall freely identify $X$ with its image under this functor. In particular, we often drop the diamond $-^\diamond$ from notation. For example, we just write $X_v$ for the $v$-site $X^\diamondsuit_v$.

For any formal scheme $\mathfrak S$ over $\Spf(K^+)$, we denote by $\mathfrak S^\ad_{\eta}$ the adic generic fibre of \cite[\S2.2]{ScholzeWeinstein}. If $\mathfrak S$ is locally of topologically finite type over $\Spf(K^+)$, this is a rigid space over $\Spa(K,K^+)$ in the sense defined above. We say that a rigid space has good reduction if it is isomorphic to such an adic generic fibre.

\section{The topological torsion subsheaf}\label{s:def-of-tt}\label{s:profinite-and-locally-constant-sheaves}

The purpose of this section is to introduce a new technical device that plays a key role in this article: the topological torsion subgroup of a $v$-sheaf.

\subsection{Locally profinite $v$-sheaves}
We start by recalling the definition of locally constant sheaves, and record some basic lemmas for later reference. For this it is beneficial to work in greater generality than in the introduction, over  general non-archimedean fields $(K,K^+)$. We first adapt \cite[Example~11.12]{etale-cohomology-of-diamonds} to this setting.
\begin{Definition}
	Let $X$ be an adic space over $\Spa(K,K^+)$. Then we have a natural injective map $X(K,K^+)\hookrightarrow |X|$ sending a morphism $x:\Spa(K,K^+)\to X$ to the image of the unique closed point under $x$. We use this to endow $X(K,K^+)$ with the subspace topology.
\end{Definition}
\begin{Definition}\label{d:profinite-sheaves}
	For any locally profinite set $G$ we define a \textbf{locally profinite $v$-sheaf} $\underline{G}$:
	\begin{enumerate}
		\item If $G$ is a discrete topological space, then we associate to $G$ the locally constant sheaf $\underline{G}$ on $\Perf_{K,v}$ that is the sheafification of the presheaf $X\mapsto G$. This is represented by the adic space $\sqcup_G\Spa(K,K^+)$.
		
		\item If $G=\varprojlim_{i\in I} G_i$ is a profinite set where the $G_i$ are finite sets considered as discrete topological spaces, then we associate to $G$ the $v$-sheaf of sets $\underline{G}:=\textstyle\varprojlim_{i\in I} \underline{G_i}$.
		By \cite[Proposition~2.4.5]{ScholzeWeinstein}, this is represented by the affinoid adic space
		\[\underline{G}=\Spa(\Map_{\cts}(G,K),\Map_{\cts}(G,K^+))\sim \textstyle\varprojlim_{i\in I}\underline{G_i}.\]
		This is sheafy as by  Lemma~\ref{l:points-of-underlineH} below, any cover of $|\underline{G}|$ admits a disjoint refinement.
		\item In both the profinite and the discrete case, we have a natural homeomorphism $|\underline{G}|=|\Spa(K,K^+)|\times G$ by Lemma~\ref{l:points-of-underlineH} below.
		Simultaneously generalising both 1 and 2, we can therefore extend the definition to locally profinite sets $G$ by glueing.
		
	\end{enumerate}
If $G$ is a locally profinite group, then $\underline{G}$ inherits the structure of an adic group: This is immediate from the observation that the functor $\underline{\phantom{-}}$ commutes with finite products.
\end{Definition}
The following lemma was used in the definition:
\begin{Lemma}\label{l:points-of-underlineH}
	Let $G$ be a locally profinite set, then we have natural homeomorphisms
	\[ |\underline{G}|=|\Spa(K,K^+)|\times G \quad \text{and} \quad \underline{G}(K,K^+)=G.\]
\end{Lemma}
\begin{proof}
	We can reduce to profinite sets $G=\varprojlim G_i$, for which we have
	\[|\underline{G}|=\varprojlim |G_i| \quad \text{and} \quad \underline{G}(K,K^+)=\varprojlim G_i(K,K^+).\]
	This reduces us to the case of finite sets, where the statements are clear.
\end{proof}

\begin{Lemma}\label{l:underline-is-left-adjoint}
	Let $H$ be a locally profinite set, and let $G$ be any adic space over $(K,K^+)$. Then we have
	\[ \Mor_K(\underline{H},G)=\Mapc(H,G(K,K^+)),\]
	functorially in $H$ and $G$.
	If $H$ is a locally profinite group and $G$ is an adic group for which $G(K,K^+)$ is a topological group, this induces
	\[ \Hom(\underline{H},G)=\Hom_{\cts}(H,G(K,K^+)).\]
\end{Lemma}
\begin{proof}
	The second part of the lemma will follow from the first by functoriality.
	
	For the first part, evaluation on $(K,K^+)$-points defines by Lemma~\ref{l:points-of-underlineH} a morphism
	\[ \Map(\underline{H},G)\to \Map_{\cts}(H,G(K,K^+)).\]
	Here continuity is clear as $G(K,K^+)$ has the subspace topology of $|G|$, and $H$ that of $|\underline{H}|$.
	
	The map is clearly functorial in $H$ and $G$.
	To see that it is a bijection,
	we can therefore by glueing reduce to the case that $H$ is profinite and $G$ is affinoid.
	
	Given a continuous map $\varphi:H\to G(K,K^+)$, there is an associated map of $K^+$-algebras
	\[ \O^+(G)\to \Map_{\cts}(H,K^+)\]
	that interprets $f\in \O^+(G)$ as a map $f:G\to \B^1:=\Spa(K\langle X\rangle,K^+\langle X\rangle)$ and sends it to
	\[H\xrightarrow{\varphi} G(K,K^+)\xrightarrow{f} \B^1(K,K^+)=K^+.\]
	This is clearly functorial in $H$ and $G$.
	We have thus defined natural morphisms
	\[ \Map_{\cts}(H,G(K,K^+)) \to \Map(\underline{H},G)\to \Map_{\cts}(H,G(K,K^+)).\]
	It is now formal from the functoriality that these compose to the identity: We can first reduce to the case $H=\{x\}$ and then to $G=\Spa(K,K^+)$, for which the statement is clear.

	It remains to prove that the second map in the above composition is injective. For this, consider the topological space  $H'$ whose underlying set is that of $H$, but endowed with the discrete topology. There is by functoriality  a natural map $\underline{H'}\to \underline{H}$, which on global sections
	\[\O^+(\underline{H})=\Map_{\cts}(H,K^+)\to \O^+(\underline{H'})=\Map(H',K^+)\]
	is injective. Consequently, so is the pullback $\Map(\underline{H},G)\to \Map(\underline{H},G)$ for affinoid $G$.
	We are therefore reduced to the case of discrete $H$, where the lemma is clear.
\end{proof}

\subsection{Topological torsion sheaves}\label{s:tt-sheaves}
From now on, let $K$ be a perfectoid field, and $K^+\subseteq K$ a ring of integral elements as before. We work in the category of $v$-sheaves on $\Perf_K$. For a locally profinite set $G$, to simplify notation, we will from now on also denote just by $G$ the pro-constant $v$-sheaf $\underline{G}$ from Definition~\ref{d:profinite-sheaves} when it is clear from the context what is meant.
\begin{Definition}\label{d:tt}
	Let $F$ be an abelian $v$-sheaf on $\Perf_K$.
	\begin{enumerate}
\item  Consider the internal Hom sheaf $\HOM(\whZ,F)$ where \[\whZ:=\underline{\whZ}=\varprojlim_{N\in\N} \underline{\Z/N\Z}\] is the profinite sheaf of Definition~\ref{d:profinite-sheaves}.
There is a  natural evaluation map at $1\in \whZ$:
\[ \mathrm{e}_F:\HOM(\whZ,F)\to F.\]
We define the \textbf{topological torsion subsheaf} of $F$ to be the abelian subsheaf
\[ F^\tt=\im(\mathrm{e}_F)\subseteq F.\]
\item We say that a $v$-sheaf $F$ is \textbf{topological torsion} if  $F^\tt\to F$ is an isomorphism.
\item We say that $F$ is \textbf{strongly topological torsion} if already $\mathrm e_F$ is  an isomorphism.
\item We say that $F$ is \textbf{topologically torsionfree} if $F^\tt=0$.
\end{enumerate}
\begin{Definition}
We make analogous definitions for the \textbf{topologically $p$-torsion subsheaf}
\[F\langle p^\infty\rangle:=\im\big(\HOM(\Z_p,F)\to F\big)\subseteq F\]
by replacing $\whZ$ with $\Z_p$. 
Via the natural projection $\whZ\twoheadrightarrow \Z_p$, we see that $F\langle p^\infty\rangle\subseteq F^{\tt}\subseteq F$.
In particular, if $F$ is topologically $p$-torsion, it is automatically topologically torsion.
\end{Definition}
\begin{Example}\label{ex:top-tor-in-G_m}
	As we will see in \cref{s:tt-of-rigid-groups}, for the $v$-sheaf represented by the rigid group $\G_m$, the sheaf $\G_m\langle p^\infty\rangle\subseteq \G_m$ is given by the open unit disc $\wh{\G}_m\subseteq \G_m$ of radius 1 around 1, and $\G_m^\tt$ is the open subgroup given by the translates by all roots of unity $\mu_N\subseteq \G_m$ for $N\in \N$.
\end{Example}

\begin{Remark}
	The construction of $F\langle p^\infty\rangle$ may be seen as an analogue for $v$-sheaves of a related construction by Fargues  \cite[\S1.6, \S2.4]{Fargues-groupes-analytiques} for rigid groups in characteristic $0$  which gives rise to analytic $p$-divisible groups in the sense of Fargues, cf \S\ref{s:tt-of-rigid-groups}. Hence our notation.
\end{Remark}
\end{Definition}

\begin{Lemma}\label{l:tt-left-exact}
	For any abelian $v$-sheaf $F$, the sheaf $F^\tt$ is itself topologically torsion.
	Sending $F\mapsto F^\tt$ thus defines a right-adjoint to the forgetful functor from topological torsion abelian $v$-sheaves to abelian $v$-sheaves.
	 The analogous statements hold for $F\tpt$.
\end{Lemma}

\begin{proof}
	It suffices to show that the natural map $(F^{\tt})^\tt\to F^{\tt}$ is an isomorphism. It is then formal that this defines the unit of an adjunction with co-unit given by $F^\tt\to F$.
	
	It suffices to see that the map is surjective, for which it suffices to prove that the map
$\HOM(\whZ,F^{\tt})\to F^{\tt}$
	is surjective. For this it suffices to prove that the top map in the diagram
	\[\begin{tikzcd}
		{\HOM(\whZ,\HOM(\whZ,F))} \arrow[d] \arrow[r] & {\HOM(\whZ,F)} \arrow[d,two heads] \\
		{\HOM(\whZ,F^\tt)} \arrow[r] & F^{\tt}
	\end{tikzcd}
	\]
	is surjective. This holds because for any perfectoid space $T$ over $K$ and any $T$-linear homomorphism $\varphi:\whZ\times T\to F$, a preimage is given by the map $\whZ\times \whZ\times T\to F$ defined by $(a,b,x)\mapsto \varphi(ab,x)$.
	Thus $F^\tt$ is topologically torsion. The case of $F\langle p^\infty\rangle$ is analogous.
\end{proof}
\begin{Example}
	The reason why we need  both the notion of topological torsion and strongly topological torsion sheaves is that each has its advantages and pathologies:
	\begin{enumerate}
		\item For an example where $F\langle p^\infty\rangle \neq \HOM(\Z_p,F)$, consider the $v$-sheaf $F=\Z_p/\Z$ in the exact sequence
		\[ 0\to \Z\to \Z_p\to\Z_p/\Z\to 0.\]
		One can show that $\HOM(\Z_p,F)=\underline{\Q_p}$ \cite[\S2.3]{heuer-isoclasses}, but nevertheless $F\langle p^\infty\rangle=F$.
		This also shows that the category of topologically torsion sheaves is not abelian, that subsheaves of topological torsion sheaves need not be topological torsion, and that in contrast to $\HOM(\Z_p,-)$, the functor $-\langle p^\infty\rangle$ is neither left-exact nor right-exact.
		\item On the other hand, the analogue of Lemma~\ref{l:tt-left-exact} does not hold for $\HOM(\whZ,-)$, namely
		\[ \HOM(\whZ,\HOM(\whZ,-))\to \Hom(\whZ,-)\]
		is not an equivalence: Indeed, by the Yoneda Lemma and the tensor-hom adjunction this follows from the fact that $\whZ\to \whZ\otimes_{\Z}\whZ$ is not an isomorphism of $v$-sheaves.
	\end{enumerate}
 \end{Example}
We will see in \cref{s:tt-of-rigid-groups} below that such pathologies cannot occur for $v$-sheaves represented by rigid spaces. But before, we first prove some useful lemmas that hold in greater generality:

\begin{Lemma}\label{l:coker-Ftt->F-torsionfree}
	\begin{enumerate}
	\item The cokernel of $F^\tt\to F$ is torsionfree. 
	\item The cokernel of $F\tpt\to F$ is $p$-torsionfree.
	\end{enumerate}
\end{Lemma}
\begin{proof}
	The $v$-sheaf $\whZ/\Z$ is uniquely divisible by a Snake Lemma argument. Thus also $\EXT^1_v(\whZ/\Z,F)$ is uniquely divisible. The statement follows since we have an exact sequence
	\[ \HOM(\whZ,F)\to \HOM(\Z,F)\to \EXT^1_v(\whZ/\Z,F).\]
	Similarly for the $v$-sheaf $\Z_p/\Z$, which is uniquely $p$-divisible.
\end{proof}

\begin{Lemma}\label{l:preserve-inj}
	The functors $\HOM(\whZ,-)$ and $\HOM(\Z_p,-)$ preserve injectives in the category of abelian sheaves on $\Perf_{K,v}$.
\end{Lemma}
\begin{proof}
	Let $F\to G$ be an injection of abelian $v$-sheaves and let $Q$ be an injective $v$-sheaf. Then $F\otimes_{\Z} \whZ\to G\otimes_{\Z} \whZ$ is injective as $\whZ$ is a torsionfree $v$-sheaf and thus flat over $\Z$. Hence
	\[\begin{tikzcd}[row sep =0.2cm]
		{\Hom(G,\HOM(\whZ,Q))} \arrow[r] \arrow[d,equal] & {\Hom(F,\HOM(\whZ,Q))} \arrow[d,equal] \\
		{\Hom(\whZ\otimes_{\Z} G,Q)} \arrow[r] & {\Hom(\whZ\otimes_{\Z} F,Q)}
	\end{tikzcd}\]
	is surjective since $Q$ is injective.
\end{proof}

\begin{Lemma}\label{l:cohomology-of-stt-is-tt}
	Assume that $F$ is strongly topologically ($p$-)torsion. Let $\pi:X\to \Spa(K,K^+)$ be a locally spatial diamond. Then $R^1\pi_{\ast\tau}F$ is topologically ($p$-)torsion for $\tau\in \{\et,v\}$.
\end{Lemma}
\begin{proof}
	There is a natural equivalence of functors on $v$-sheaves
	$\pi_{\ast}\HOM(\whZ,-)=\HOM(\whZ,\pi_{\ast}-)$.
	By \cref{l:preserve-inj} this induces Grothendieck spectral sequences which yield a natural map
	\[ R^1\pi_{\ast}F=R^1\pi_{\ast}\HOM(\whZ,F)\to R^1(\pi_{\ast}\HOM(\whZ,-))(F)\to  \HOM(\whZ,R^1\pi_{\ast}F).\]
	Since the above equivalence is also natural in the first entry $\whZ$, it is compatible with the evaluation map, which means that the above map is a splitting of the evaluation map
	\[\HOM(\whZ,R^1\pi_{\ast}F)\to R^1\pi_{\ast}F.\]
	In particular, the latter map is surjective. The same argument works for $\Z_p$.
\end{proof}

\subsection{Topologically torsion subgroups of rigid groups}\label{s:tt-of-rigid-groups}
Assume from now on that $K$ is a perfectoid field extension of $\Q_p$, and that the $v$-sheaf $F$ is represented by a rigid group $G$. We recall that by our conventions, this means a group object in the category of adic space of locally finite presentation over $\Spa(K,K^+)$. In this case, the topologically torsion subgroup $G\langle p^\infty\rangle$ can be described explicitly, and is often represented by an analytic \mbox{$p$-divisible} subgroup in the sense of Fargues: Namely, we prove the following extension of a result of Fargues \cite[Th\'eor\`eme~1.2]{Fargues-groupes-analytiques} which shows that our notion of topological torsion subsheaves generalises his notion of $p$-divisible analytic subgroups from rigid groups (in the sense of Tate) to general $v$-sheaves:
\begin{Proposition}\label{l:hat-vs-p-divisible-groups}
 Let $G$ be a rigid group over $(K,K^+)$. Then:
		\begin{enumerate}
		\item  We have
		$G\tpt=\HOM(\Z_p,G)\subseteq G$, and this subsheaf is represented by a rigid open subgroup that is strongly topologically torsion. Its $K$-points are characterised by
		\[ \HOM(\Z_p,G)(K,K^+)=\{x\in G(K,K^+)\mid p^n\cdot x\xrightarrow{n\to \infty} 0\}.\]
		\item This open subgroup fits into a left-exact sequence of $v$-sheaves
		\[0\to G[p^\infty]\to G\tpt\xrightarrow{\log}\mathrm{Lie}(G)\otimes_K \G_a.\]
		If $[p]:G\to G$ is surjective, this sequence is also right-exact.
		\item We have $G^{\tt}=\HOM(\whZ,G)$ and this is represented by the open subgroup of $G$ 
		\[G^{\tt}=\bigcup_{(N,p)=1} [N]^{-1}(G\langle p^\infty\rangle).\]
		If $K$ is algebraically closed, we can more explicitly describe this as the union
		\[G^{\tt}=G\langle p^\infty\rangle \cdot \bigcup_{(N,p)=1} G[N](K,K^+)\subseteq G.\]
		\item Assume that $G$ has good reduction, i.e.\ there is a formal group scheme $\mathfrak G$ that is flat of topologically finite presentation  over $K^+$ such that $\mathfrak G^{\ad}_{\eta}=G$. 
		Let $\mathfrak G[p^\infty]$ be the sheaf on $p$-nilpotent $K^+$-algebras $R$  defined by 
		$\mathfrak G[p^\infty](R)=\textstyle\varinjlim_{n\in \N} \mathfrak G[p^n](R)$.
		Then 
		\[G\langle p^\infty\rangle=\mathfrak G[p^\infty]^{\ad}_{\eta}\] is the adic generic fibre in the sense of Scholze--Weinstein \cite[\S2.2]{ScholzeWeinstein}. In particular, $G\langle p^\infty\rangle(K,K^+)$ (respectively, $G^\tt(K,K^+)$) consists of those points of $\mathfrak G(K^+)$ which reduce to a $p$-torsion point (respectively, torsion point) on the special fibre.
	\end{enumerate}
\end{Proposition}

\begin{Remark}
	The aforementioned Theorem of
Fargues asserts that for a (classical) rigid group over $(K,\O_K)$ of characteristic $0$, there exists a rigid open subgroup $U\subseteq G$ such that \[|U|_{Ber}=\{x\in|G|_{Ber}\mid p^n\cdot x\xrightarrow{n\to \infty}0 \}.\]
\cref{l:hat-vs-p-divisible-groups} recovers $U$ and shows that it is equal to $G\langle p^\infty \rangle$. Fargues also shows part 2 in this case, and part 4 is his point of view on the $p$-divisible group of $G$  \cite[\S 0]{Fargues-groupes-analytiques}.

 From this perspective, the new aspects of \cref{l:hat-vs-p-divisible-groups} in the special case of classical rigid groups over $(K,\O_K)$ are the description of $U$ in terms of the $v$-sheaf $\underline{\Hom}(\uZp,-)$ (which crucially uses the language of $v$-sheaves), and the extension to coprime torsion via part 3.
\end{Remark}
\begin{Remark}
	In the second part of Proposition~\ref{l:hat-vs-p-divisible-groups}.3, we form the disjoint union in the category of adic spaces. This is relevant because in the setting of classical rigid spaces in the sense of Tate, the subspaces $xG\tpt$ for $x\in G[N](K)$ typically form a set-theoretic cover of $G$ that is not admissible. In the setting of adic spaces, this cover misses some points of rank $2$ in $G$ that are intentionally not included in $G^\tt$. See also \cref{s:abeloid-tt} below.
\end{Remark}

\begin{Example}\label{ex:interpretation-of-wh B-for-B-of-good-reduction}
	\begin{enumerate}
		\item\label{enum:example-G_m} To get an explicit description of $G\langle p^\infty\rangle$, it suffices by part 4 to find an open subgroup of good reduction $G_0\subseteq G$ such that $[p]^{-1}(G_0)\subseteq G_0$.  For example, for $G=\G_m$, such an open subgroup is given by the affine torus $\mathfrak G=\G_{m,\O_K}$. This shows that $\G_m\langle p^\infty\rangle=\widehat{\G}_m\subseteq \G_m$ is the open disc of radius 1 around 1. In particular, by part 4, we deduce that
		$\G_m^{\tt}=\mu(K)\widehat{\G}_m$,
		justifying the description in \cref{ex:top-tor-in-G_m}.
		\item More generally, such an open subgroup exists for semi-abeloid varieties $A$, which are of interest in the context of Picard functors. By Raynaud uniformisation, there is a maximal open subgroup $A^+\subseteq A$ that admits a connected smooth formal model. We then let $G_0$ be the disjoint union of translates of $A^+$ over points in $A[p^\infty]$. In fact, with this definition, the analogues of parts 1,3,4 also hold in characteristic $p$.
		\item If $G=\G_a$, then we cannot find an open subgroup of good reduction $G_0$ such that $[p]^{-1}(G_0)\subseteq G_0$.  This is consistent with the fact that the $\G_{a,\O_K}[p^n]$ do not form a $p$-divisible group. Instead, in this case, we have
		$\G_a\langle p^\infty\rangle=\G_a^{\tt}=\G_a$.
		\item The restriction to the rigid case is necessary: For example, set $G=\Z_l$ for $l\neq p$. Here $G\tpt=0 \subseteq G$ is not open, rather it is closed. 		
		For part 4, consider also the example of the formal scheme $\mathfrak G=\varprojlim_{[p]}\G_{m,K^+}$.
		Its adic generic fibre $G$ is a perfectoid group whose topological torsion subgroup is given by the open subgroup $\varprojlim_{[p]}\wh{\G}_m$. This is closely related to the universal cover of $\mu_{p^\infty}$ in the sense of \cite[\S3.1]{ScholzeWeinstein}. But the description of $G\langle p^\infty\rangle$  in terms of torsion points of the special fibre is no longer valid as $\mathfrak G$ is uniquely $p$-divisible, and hence its special fibre is $p$-torsionfree.
	\end{enumerate}
\end{Example}

\begin{Remark}
	We note that if $G$ is a rigid group, there is an alternative, ``analytic'' description of $G^\tt$, as follows. Using the $p$-adic logarithm, one can show that any morphism $\varphi:\Z_p\times T\to G$ in $\HOM(\Z_p,G)(T)$ admits an analytic continuation locally. More precisely, for $n\gg 0$, the restriction of $\varphi$ to $p^n\Z_p\times T$ extends uniquely to a $T$-linear homomorphism $\varphi:\B_n\times T\to G$
	where $\B_n\subseteq \G_a$ is the closed disc of radius $\leq |p^n|$ with its additive structure.
	
	In other words, if $\Z_p(\epsilon)$ denotes for any $\epsilon>0$ the open subgroup of $\G_a$ defined by the union of closed discs of radius $\epsilon$ around $\Z_p\subseteq \G_a(K)$, then we have
	\[ G\tpt = \varinjlim_{\epsilon>0}\HOM(\Z_p(\epsilon),G).\]
	While we do not need this observation in this article, it is useful for other applications.
\end{Remark}

In light of  Proposition~\ref{l:hat-vs-p-divisible-groups}, one can use the topological $p$-torsion subsheaf  to generalise the notion of ``analytic $p$-divisible groups'' in the sense of Fargues \cite[\S2 D\'efinition~2]{Fargues-groupes-analytiques}:

\begin{Definition}\label{d:analytic-divisible-group}
	\begin{enumerate}
		\item 
	We call a rigid group $G$ over $(K,K^+)$ a \textbf{$p$-divisible analytic group} if
	$G$ is topologically $p$-torsion (i.e.\ $G\tpt=G$) and  $[p]:G\to G$ is surjective.
	
	\item 
	We call a rigid group $G$ an \textbf{analytic divisible group} if
	$G$ is topologically torsion (i.e.\ $G^\tt=G$) and  $[n]:G\to G$ is surjective for all $n\in \N$.
\end{enumerate}
\end{Definition}
Examples of analytic divisible groups include $\Q/\Z$, $\G_a$ and $G^\tt$ for any rigid group $G$ for which $[n]:G\to G$ is surjective for all $n\in \N$, e.g.\ for abeloid varieties. One can then characterise $G^\tt$ as the maximal analytic divisible subgroup of $G$, and similarly for $G\tpt$.

\medskip

We now start with the proof of \cref{l:hat-vs-p-divisible-groups}. For this we use:
\begin{Lemma}\label{p:G(K)-is-complete}
	Let $G$ be any rigid group over $(K,K^+)$. Then $G(K,K^+)$ is a complete topological group (complete with respect to the uniform structure given by open subgroups).
\end{Lemma}
\begin{proof}
	To simplify notation, let us just write $G(K)$ for $G(K,K^+)$.	
	We first note that $G(K)$ is a topological group: This is not completely obvious for adic groups since the map $|G\times G|\to |G|\times |G|$ is usually not a bijection. However, the continuous bijection \[(G\times G)(K)\to G(K)\times G(K)\] on $K$-points is indeed a homeomorphism:
	This follows from the fact that any rigid group admits a neighbourhood basis of $0$ consisting of open rigid subgroups $U\subseteq G$ such that $U\cong \B^d$ as rigid spaces (\cite[\S1, Corollaire 4]{Fargues-groupes-analytiques} over $(K,\O_K)$, \cite[Corollary~3.7]{heuer-G-torsors-perfectoid-spaces} in general). This also shows that $G(K)$ is Hausdorff.
	
	To see that $G(K)$ is complete, let $(x_n)_{n\in \N}$ be any Cauchy sequence. Then for $n,m\gg 0$, we have $x_n-x_m\in U$, and thus $x_n\in x_m+U$. After translating by $x_m$, we may therefore assume $x_n\in U$ for all $n$.
	Let now $f\in \O^+(U)$ and consider the map 
	\[U\times U\to \B, \quad y,\delta\mapsto f(y+\delta)-f(y).\]
	For any $\epsilon>0$, let $V_{\epsilon}\subseteq U\times U$ be the pullback of the open ball $\B_{\epsilon}$ of radius $\epsilon$ around $0$. Since $U\times \{0\}\subseteq V_{\epsilon}$, and $U\times \{0\}$ is the tilde limit of $\varprojlim_{n\in\N} U\times \B_{n}$ where $\B_n$ is the ball of radius $|p^n|$, it follows from a standard approximation argument that there is an open subgroup $W_{\epsilon}\subseteq U$ such that $U\times W_{\epsilon}\subseteq V_\epsilon$. For $n,m\gg 0$, we then have $x_m-x_n\in W_\epsilon$ which implies $|f(x_m)-f(x_n)|\leq \epsilon $.
	Thus $f(x_m)$ is a Cauchy sequence in $K^+$. It thus makes sense to define
	\[ x:\O^+(U)\to K^+,\quad f\mapsto \mathrm{lim}_n f(x_n).\] 
	One easily checks that this defines the desired limit $x_n\to x\in U(K,K^+)\subseteq G(K,K^+)$. 
\end{proof}
\begin{Remark}
	With more work, one can show that $G(K,K^+)$ is a complete topological group for any adic group  $G$ over $(K,K^+)$ satisfying the following mild technical assumption: $G$ admits a basis of uniform affine open subspaces $\mathcal B$ such that for each $U=\Spa(A,A^+)$ and $V=\Spa(B,B^+)$ in $\mathcal B$, the fibre product $U\times_KV$ exists and is affinoid with $\O(U\times_KV)=A\hotimes_K B$ the tensor product of complete normed $K$-vector spaces. For example, this holds for perfectoid groups. One then obtains the analogous description of $K$-points of $G\tpt$.
\end{Remark}	

\begin{proof}[Proof of \cref{l:hat-vs-p-divisible-groups}]
We begin by describing the $(K,K^+)$-points of $G\langle p^\infty\rangle$ and $G^\tt$:
	By Lemma~\ref{l:underline-is-left-adjoint}, we have $\HOM(\whZ,G)(K)=\Hom_{\cts}(\whZ,G(K))$.
Let  $\varphi:\whZ\to G(K)$ be any continuous homomorphism. This is uniquely determined by the image $x=\varphi(1)$ of  $1$. By continuity, since $n!\to 0$ in $\wh\Z$, we have $x^{n!}\to 0$ in $G(K)$.

	Conversely, let $x\in G(K)$ be such that $x^{n!}\to 0$. Let $f:\Z\to G(K)$ be the homomorphism defined by $f(1)=x$. For any open subgroup $U\subseteq G(K)$, we have $x^{n!}\in U$ for $n\gg 0$ and thus $n!\Z\subseteq f^{-1}(U)$, so $f$ is continuous for the subspace topology on $\Z\subseteq \whZ$. By the universal property of the completion, $f$ now extends to $\whZ$ as $G(K)$ is complete by \cref{p:G(K)-is-complete}.
	
	The case of topological $p$-torsion is analogous. This gives the description of $G\tpt(K,K^+)$.

By \cite[Proposition~3.5, Corollary~3.8]{heuer-G-torsors-perfectoid-spaces}, there exists an open subgroup $G_1\subseteq G$ of good reduction that admits an open immersion $\log:G_1\hookrightarrow \mathrm{Lie}(G)\otimes_K \G_a$. Let \[G_2:=\textstyle\bigcup_{n\in\N} [p]^{-n}(G_1)\subseteq G,\] this is clearly an open adic subgroup of $G$. We claim that $G_2=G\tpt$.
	
To see that $G\tpt\subseteq G_2$, let $T=\Spa(S,S^+)$ be affinoid perfectoid over $(K,K^+)$ and let \[\varphi:\uZp\times T\to G\times T\] be any homomorphism over $T$. Then the preimage of $G_1\times T$ contains some neighbourhood $p^n\uZp\times T$ of the identity because 
$|\Z_p\times T|=\varprojlim |\Z/p^n\Z\times T|=\Z_p\times |T|$.
By the definition of $G_2$, 
this implies that $\varphi$ factors through $G_2$.

To see the converse, let now $\mathfrak G$ be the formal model of $G_1$ and consider the open subgroup 
\[G_0:=\mathfrak G[p^\infty]^{\ad}_\eta\subseteq G_1\subseteq G_2\subseteq G.\]
By \cite[Proposition~2.2.2.(ii)]{ScholzeWeinstein}, the points of $G_0$ on perfectoid $(K,K^+)$-algebras $(R,R^+)$ are given by the sheafification of the presheaf
\[(R,R^+)\mapsto \{x\in \mathfrak G(R^+)|\text{ for any $n\in \N$ we have }x\bmod p^n\in \mathfrak G[p^\infty](R^+/p^n)\}.\]

Suppose first that we are given any point $\varphi\in G_0(T)$. To prove that this is contained in $G\tpt(T)\subseteq G(T)$, we may without loss of generality pass from $T$ to an analytic cover. By the above, we may therefore assume that $\varphi$ has a formal model, i.e.\ comes from a group homomorphism
$\varphi:\Spf(S^+)\to \mathfrak G$
that reduces mod $p^n$ to compatible homomorphisms \[\uZp\times \Spf(S^+/p^n)\to \mathfrak G_{K^+/p^n}[p^\infty]\] for all $n$. Taking the limit over $n$ and passing to the generic fibre, this gives the desired map
\[ \uZp\times T\to G_0=\mathfrak G[p^\infty]^{\ad}_{\eta}.\]

Finally, we observe that given any $x\in G_2(T)$, there is $n$ such that $p^nx\in G_0(T)$. Let $\varphi:\underline{\Z}\times T\to G$ be the homomorphism sending $1\mapsto x$. As we have just seen, its restriction to $p^n\underline{\Z}\times T$ extends uniquely to a map $\widehat{\varphi}:p^n\underline{\Z}_p\times T\to G_0$. Using the short exact sequence 
\[ 0\to p^n\underline{\Z}\to p^n\underline{\Z}_p\times \underline{\Z}\to \underline{\Z}_p\to 0\]
of $v$-sheaves, $\varphi$ now extends uniquely to $\underline{\Z}_p\times T$ by the universal property of the cokernel.

This proves part 1. Part 2 follows from the open immersion $G_1\hookrightarrow \mathrm{Lie}(G)\otimes_K \G_a$, which can be uniquely extended to $G_2$ by sending $x\in G_0(T)$ with $p^nx\in G_1(T)$ to $\log(x)=p^{-n}\log(p^nx)$. The statement about surjectivity follows  as $\mathrm{Lie}(G)\otimes_K \G_a=\cup_{n\in \N}p^{-n}\log(G_1)$.

For part 4, it remains to prove that any morphism $\varphi:\underline{\Z}_p\times T\to G$ factors through $G_0\subseteq G$. Since this is an open subgroup, we see as in the beginning of the proof that there is $n$ such that $\varphi(p^n\underline{\Z}_p\times T)\subseteq G_0$. But in the case of good reduction, we also have $[p]_G^{-1}(G_0)=G_0$ by the above explicit description of $G_0=\mathfrak G[p^\infty]_{\eta}^{\ad}$. Thus $\varphi(\underline{\Z}_p\times T)\subseteq G_0$.

To deduce part 3, write $\whZ=\Z_p\times \whZ^{(p)}$ where $\whZ^{(p)}=\varprojlim_{(N,p)=1} \Z/N$. It suffices to show:
\begin{Lemma}\label{l:coprime-to-p-HOM}
Let $H=\varprojlim_{i\in I} H_i$ be any profinite group for which the order of each $H_i$ is coprime to $p$. Let $G$ be any rigid group over $(K,K^+)$. Then we have
\[\HOM({H},G)=\textstyle\varinjlim_{i\in I}\HOM(H_i,G).\]
\end{Lemma}
\begin{proof}
Let $f:{H}\times T\to G$ be any homomorphism of adic groups over a perfectoid space $T$. We need to prove that some open subgroup of $H$ gets sent to the identity. In the notation introduced above, we can therefore assume that $f$ factors through the open subgroup $G_0=\mathfrak G[p^\infty]^{\ad}_\eta\subseteq G$ and that $f$ admits a formal model ${H}\times \Spf(S^+)\to \mathfrak G$. We claim that $f$ is then trivial. To see this, we may reduce to $T=\Spa(K,K^+)$.
Then $f$ corresponds to an inverse system of homomorphisms
$H\to \mathfrak G(\O_K/p^n)[p^\infty]$.
Since the right hand side is discrete, this factors through $H_i$ for some $i$. But $H_i$ is of order coprime to $p$, so this map is trivial.
\end{proof}
This finishes the proof of Proposition~\ref{l:hat-vs-p-divisible-groups}.
\end{proof}

\subsection{Topological torsion in abeloid varieties}\label{s:abeloid-tt}
An example of particular interest for our applications to the $p$-adic Simpson correspondence is the case of  abeloid varieties, i.e.\ connected smooth proper rigid groups. For these we have the following curious special case:

\begin{Proposition}\label{l:tt-over-C_p}
	Let $A$ be an abeloid variety over $K=\C_p$ or $\C_p^\flat$.  Then \[A^\tt(K)=A(K).\]
\end{Proposition}
\begin{Remark}
	This does \textit{not} mean that the open immersion $A^\tt\to A$ is an isomorphism. Indeed, this is not the case; for example $A$ is connected while $A^\tt$ is an infinite disjoint union of open discs  by Proposition~\ref{l:hat-vs-p-divisible-groups}.
	A closely related phenomenon is that the statement of \cref{l:tt-over-C_p} becomes false for any non-trivial extension of $K$, already for Tate curves.
	
	This difference is relevant to the $p$-adic Simpson correspondence as it explains why the description in \cref{t:intro-Corollary-Cp} in terms of vanishing Chern classes is only valid over $\C_p$.
\end{Remark}
\begin{proof}
	If $A$ has good reduction with special fibre $\overline{A}$ over the residue field $\overline{\F}_p$, then by \cref{l:hat-vs-p-divisible-groups}.4, the specialisation map $A(K)\to \overline{A}(\overline{\F}_p)$ induces a short exact sequence
	\begin{equation}\label{eq:tpt-ses-of-abeloid-w-good-reduction}
	0\to A\tpt(K)\to A(K)\to \overline{A}(\overline{\F}_p)\tf\to 0.
	\end{equation}
	The last term is an abelian torsion group: Indeed, the abelian variety $\overline{A}$ is already defined over a finite extension $\F_q$ of $\F_p$ in $\overline{\F}_p$, and hence \[\overline{A}(\overline{\F}_p)=\textstyle\varinjlim_{k|\F_q} \overline{A}(k)\] is a colimit of finite groups, where $k$ ranges through finite extensions of $\F_q$ in $\overline\F_p$. By the description of $A^\tt$ in \cref{l:hat-vs-p-divisible-groups}.3, this shows that $A^\tt(K)=A(K)$ in this case.
	
	The general case follows from the case of good reduction via Raynaud uniformisation: By \cite[Theorem 7.6.4]{Lutkebohmert_RigidCurves}, we can write  $A=E/M$ for some analytic short exact sequence \[0\to T\to E\to B\to 0\] of rigid groups where $B$ is an abeloid variety of good reduction and $T$ is a rigid torus, and where $M\subseteq E(K)$ is a discrete lattice of rank equal to the rank of $T$. It is easy to see that this exact sequence of $E$ becomes split over $B\tpt$, hence induces an exact sequence
	\begin{equation}\label{eq:tpt-of-Raynaud-extension}
	0\to T\tpt\to E\tpt\to B\tpt\to 0
	\end{equation}
	 Let $M\cdot  E\tpt\subseteq E$ be the open subgroup generated by $M$ and $E\tpt$. The composition
	\[M\cdot  E\tpt\to E\to A\]
 factors through $A\tpt$. Recall that the cokernel of $A\tpt\to A^\tt$ is torsion by \cref{l:hat-vs-p-divisible-groups}.3, so to prove that $A^\tt(K)=A(K)$ it now suffices to prove that the above map is surjective on $K$-points after tensoring with $\Q$. Hence it suffices to prove the following:
\end{proof}
\begin{Lemma}\label{l:whE(K)-vs-E(K)-over-C_p}
	Let $A=E/M$ be an abeloid variety over $\C_p$ or $\C_p^\flat$. Then \[(E\tpt(K)\cdot M)\otimes \Q=E(K)\otimes \Q.\]
\end{Lemma}
\begin{proof}
	As we have already noted above, the group $\overline B(\overline{\F}_p)$ is torsion, and thus 
	$\overline B(\overline{\F}_p)\otimes \Q=0$.
	By tensoring with $\Q$ the short exact sequence \eqref{eq:tpt-ses-of-abeloid-w-good-reduction} in the case of $A=B$,	this proves  the statement in the case of good reduction, where the lattice $M$ is trivial.
	
	The same argument shows that the inclusion 
	$(1+\m_K)\otimes \Q\to \O_{\C_p}^\times\otimes \Q$
	is an isomorphism as $\overline{\F}_p^\times$ is torsion. Consequently, the valuation on $K$ defines a short exact sequence
	\[ 0\to (1+\m_K)\otimes \Q\to K^\times\otimes \Q\to \Q\to 0\]
	If $M^\vee=\Hom(T,\G_m)$ denotes the character lattice of $T$, we obtain from this and the pairing $M^\vee\times T(K)\to K^\times$  an exact sequence
	\[0\to T\tpt(K)\otimes \Q\to T(K)\otimes \Q\rightarrow \Hom(M^\vee,\Q)\to 0.\]
	Now since
	$B\tpt(K)\otimes \Q=B(K)\otimes \Q$, we can use \cref{eq:tpt-of-Raynaud-extension} to extend this to an exact sequence
	\[0\to E\tpt(K)\otimes \Q\to E(K)\otimes \Q\rightarrow \Hom(M^\vee,\Q)\to 0.\]
	Finally, we consider the composition 
	\[M\to E(K)\to  \Hom(M^\vee,\Q)\]
	which is injective by duality theory of abeloids (see \cite[\S6.3, Proposition 6.1.8.c(ii)]{Lutkebohmert_RigidCurves}). Consequently, tensoring with $\Q$ turns it into an injective homomorphism of $\Q$-vector spaces of the same dimension $\rk M=\rk T=\rk M^\vee$, and hence into an isomorphism. This shows that $E\tpt(K)$ and $M$ generate all of $E(K)$ after tensoring with $\Q$.
\end{proof}

\section{Pro-finite-\'etale line bundles in families}
Let $K$ be an algebraically closed non-archimedean field over $\Q_p$ with ring of integral elements $K^+$. Let $\pi:X\to \Spa(K,K^+)$ be a connected smooth proper rigid space over $K$.
Based on the preparations of the last section as well as the prequel articles \cite{heuer-diamantine-Picard,heuer-v_lb_rigid}, the goal of this section is to answer the following question in the case of line bundles:
\begin{Question}\label{q:profet-vb}
	Which vector bundles on $X$ are trivialised by pro-finite-\'etale covers of $X$?
\end{Question}

As explained in detail in the introduction, this question is of interest in the $p$-adic Simpson correspondence, as it is known to lead to the answer of \cref{q:condition-Q} in various situations, including the case of Higgs field $0$, the case of line bundles, and of abelian varieties \cite{wuerthen_vb_on_rigid_var}\cite{MannWerner_LocSys_p-adVB}\cite{heuer-v_lb_rigid}\cite{HMW-abeloid-Simpson}.
Our aim  is to give a complete answer to \cref{q:profet-vb} in the case of line bundles, in terms of the topological torsion subgroup of the Picard variety. This gives a more geometric, more explicit, and more conceptual description than the one we previously gave in \cite[\S5]{heuer-v_lb_rigid}. Moreover, it naturally leads to a rigid moduli space of pro-finite-\'etale line bundles, without assuming that the rigid Picard functor is representable.

\subsection{The rigid analytic character variety of $\G_m$}\label{s:character-variety}
Our description of pro-finite-\'etale $v$-line bundles will be in terms of the $p$-adic analogue of the character variety, which we can define by an application of \cref{l:hat-vs-p-divisible-groups}:
\begin{Definition}\label{d:char-var}
	Let $X$ be any connected smooth proper rigid space over $(K,K^+)$. Let $x\in X( K)$ be a base point and let $\pi_1(X):=\pi^{\et}_1(X,x)$ be the \'etale fundamental group of $X$, a profinite group. Then the \textbf{continuous character variety} of $X$ is the $v$-sheaf on $\Perf_{K,v}$
	\[\HOM(\pi_1(X),\G_m)\]
	where $\pi_1(X)=\underline{\pi_1(X)}$ is considered as a profinite $v$-sheaf (\cref{d:profinite-sheaves}).
	Explicitly, the displayed sheaf sends $T\in \Perf_K$ to the set of $T$-linear homomorphisms $\underline{\pi_1(X)}\times T\to \G_m$, which we easily verify to be in bijection with the set of continuous group homomorphisms \[\Hom_{\cts}(\pi_1(X),\O^\times(T)).\]
\end{Definition}
\begin{Lemma}\label{l:cts-char-variety}
	The continuous character variety is representable by a rigid analytic group whose identity component is non-canonically isomorphic to $\wh{\G}_m^d$ for $d=\dim_{\Q_p}H^1_{\et}(X,\Q_p)$ (here $\wh{\G}_m$ is the open unit disc around $1$). Its group of connected components is torsion.
\end{Lemma}
If $\pi_1(X)$ is the profinite completion of a finitely generated group (e.g.\ if $X$ is projective), then the proof shows  that $\HOM(\pi_1(X),\G_m)$ is isomorphic to a finite disjoint union of copies of $(\G_m^{\tt})^d$. Here the disjoint union accounts for torsion in the N\'eron--Severi group.
\begin{proof}
	Any continuous character $\pi_1(X)\to \O^\times(Y)$ for any perfectoid space $Y$ factors through the maximal abelian quotient $\pi^{\mathrm{ab}}$ of $\pi_1(X)$, so we have
	\[\HOM(\pi_1(X),\G_m)=\HOM(\pi^{\mathrm{ab}},\G_m).\]
	 Let $T$ be the maximal torsionfree pro-$p$-quotient of $\pi^{\mathrm{ab}}$. By \cite[Lemma~4.11]{heuer-v_lb_rigid}, $T$ is a finite free $\Z_p$-module of rank $d$ (the finiteness is based on \cite[Theorem~1.1]{Scholze_p-adicHodgeForRigid}). Thus 
	\[\pi^{\mathrm {ab}}=T\times N\] for some profinite group $N=\varprojlim N_i$ which is coprime-to-$p$ up to a $p$-torsion factor. Hence
	\[ \HOM(\pi_1(X),\G_m)=\HOM(T,\G_m)\times \HOM(N,\G_m)\cong\wh \G_m^d\times \varinjlim\HOM(N_i,\G_m)\]
	by \cref{l:hat-vs-p-divisible-groups} and \cref{l:coprime-to-p-HOM}.
\end{proof}

\subsection{Diamantine Picard functors}
The second ingredient in our description of pro-finite-\'etale line bundles are the diamantine Picard functors introduced in \cite{heuer-diamantine-Picard}. We therefore begin the discussion by recalling the definition and some basic properties:
\begin{Definition}
	Let $\pi:Y\to \Spa(K,K^+)$ be any locally spatial diamond and consider the associated morphism of big \'etale sites $\pi_{\et}:\LSD_{Y,\et}\to \Perf_{K,\et}$. Then we call
\[\uP_{Y,\et}:=R^1\pi_{\et\ast}\G_m:\Perf_{K,\et}\to \mathrm{Ab}\]
the \'etale diamantine Picard functor, where $\mathrm{Ab}$ is the category of abelian groups. Explicitly, $\uP_{Y,\et}$ is the \'etale sheafification of the functor on $\Perf_K$ that sends
$T\mapsto \Pic_{\et}(Y\times T)$.

Second, there is also a $v$-Picard functor which instead parametrises $v$-line bundles, namely
\[\uP_{Y,v}:=R^1\pi_{v\ast}\G_m:\Perf_{K,v}\to \mathrm{Ab},\]
the $v$-sheafification of
$T\mapsto \Pic_{v}(Y\times T)$ where $T$ ranges through perfectoid spaces over $K$.
\end{Definition}

We are most interested in the case that $Y=X$ is a smooth proper rigid space as before. In this case,
the two main results of \cite{heuer-diamantine-Picard} about these functors are as follows: We first proved that $\uP_{X,\et}$ is the ``diamondification'' of the classical rigid analytic  Picard functor defined on smooth rigid analytic test objects \cite[Theorem~1.1]{heuer-diamantine-Picard}. In particular, if the latter is represented by some rigid group $G$ (which conjecturally is always the case, and this is known e.g.\ if $X$ is projective), then  $\uP_{X,\et}$ is represented by $G^\diamondsuit$. We therefore drop the additional $-^\diamondsuit$ used in \cite{heuer-diamantine-Picard} from notation and write $\uP_{X,\et}$ instead of $\uP^\diamondsuit_{X,\et}$.

Second, we proved the following geometric version of the main theorem of \cite{heuer-v_lb_rigid}:

\begin{Theorem}[{\cite[Theorem~2.7]{heuer-diamantine-Picard}}]\label{t:representability-of-v-Picard-functor}
	There is a natural short exact sequence  of abelian sheaves on $\Perf_{K,\et}$, functorial in $X$,
		\begin{equation}\label{seq:Picv-ses}
			0\to \uP_{X,\et}\to \uP_{X,v}\xrightarrow{\HTlog}  H^0(X,\Omega^1_X(-1))\otimes_K \G_a\to 0.
		\end{equation}
	In particular, $\uP_{X,v}$ is represented by a rigid group whenever $\uP_{X,\et}$ is. 
\end{Theorem}
Note that in terms of the language introduced in the introduction, the last term can be interpreted as the $p$-adic Hitchin base $\mathcal A$ in the context of line bundles.

\subsection{The pro-finite-\'etale universal cover}\label{s:profet-cover}
To make the connection to \cref{q:profet-vb}, we first reformulate the question slightly: Recall that we have chosen a base point $x\in X(K)$. We consider the universal pro-finite-\'etale cover of $X$, defined as the spatial diamond
\[ \wt \pi:\wt X:=\varprojlim_{X'\to X}X'\to \Spa(K)\]
where the index category is given by connected finite \'etale covers $(X',x')\to (X,x)$ with $x'\in X'(K)$ a lift of $x\in X(K)$. This is a pro-finite-\'etale $\pi_1(X):=\pi^{\et}_1(X,x)$-torsor. We refer to \cite[Definition~3.8]{heuer-v_lb_rigid} for some more background.

By \cite[Lemma 4.8]{heuer-v_lb_rigid}, we see that a vector bundle on $X$ is pro-finite-\'etale if and only if it becomes trivial after pullback to $\wt X\to X$.
Moreover, we argued in \cite[\S5.2]{heuer-v_lb_rigid} that $\wt X\to X$ plays a similar role for the $p$-adic Simpson correspondence as the complex universal cover plays in the complex Simpson correspondence: There is an equivalence of categories
\begin{alignat*}{2}
	\Big\{\begin{array}{@{}c@{}l}\text{\normalfont pro-finite-\'etale }\\\text{\normalfont $v$-vector bundles on $X$} \end{array}\Big\}&\isomarrow&& \Big\{\begin{array}{@{}c@{}l}\text{\normalfont continuous representations of $\pi_1^{\et}(X,x)$}\\\text{on finite dimensional $K$-vector spaces}\end{array}\Big\}\\
	V&\mapsto&&V(\wt X)
\end{alignat*}
since $\wt X\to X$ is a $\pi_1(X)$-torsor and $\O(\wt X)=K$. 

\medskip

In order to describe the left hand side for line bundles, we now use the diamantine Picard functor $\uP_{\wt X,\et}$ of the universal cover. Namely,
the main technical result of this article is that we can describe pro-finite-\'etale line bundles  in terms of the topological torsion subsheaf 
\[ \uP_{X,\et}^\tt:=(\uP_{X,\et})^\tt\]
as defined in \S\ref{s:def-of-tt} applied to the diamantine Picard functor of $X$, and that this moreover has a close relation to the character variety from \cref{d:char-var}.
\begin{Theorem}\label{t:tt-torsion-is-1+m-torsors}
	Let $\pi:X\to \Spa(K,K^+)$ be a connected smooth proper rigid space.
	\begin{enumerate}
	\item There is a short exact sequence of abelian sheaves on $\Perf_{K,\et}$
	\[ 0\to \uP_{X,\et}^\tt\to \uP_{X,\et}\to \uP_{\wt X,\et},\]
	in which the first term can equivalently be described as 
	$\uP_{X,\et}^\tt=R^1\pi_{\et\ast}(\G_m^\tt)$.
	\item In the $v$-topology, we similarly have a short exact sequence
	\[ 0\to \uP_{X,v}^\tt\to \uP_{X,v}\to \uP_{\wt X,v},\]
	in which the first term can be naturally identified with the character variety
	\[ \uP_{X,v}^\tt=R^1\pi_{v\ast}(\G_m^\tt)=\HOM(\pi_1(X),\G_m).\]
	\end{enumerate}
\end{Theorem}
We will prove \cref{t:tt-torsion-is-1+m-torsors} over the course of this section. Before, we note that this gives the desired answer to Question~\ref{q:profet-vb} for line bundles: The pro-finite-\'etale line bundles on $X$ are precisely the topological torsion $K$-points of the (classical) rigid-analytic Picard variety:
\begin{Corollary}\label{c:tt-descr-of-profet-line-bundles}
	Assume that the rigid analytic Picard functor of $X$ is represented by a rigid group $G$. Then there is a left exact sequence
	\[ 0\to G(K)^\tt\to \Pic_{\et}(X)\to \Pic_{\et}(\wt X).\]
	where $G(K)^\tt\subseteq G(K)$ is the subgroup of elements $x$ such that $x^{n!}\to 1$ for $n\to \infty$. 
\end{Corollary}
\begin{proof}
	By  \cite[Theorem~1.1]{heuer-diamantine-Picard}, the assumption implies that also $\uP_{X,\et}$ is also represented by the rigid group $G$. The result thus follows from \cref{t:tt-torsion-is-1+m-torsors}.1 by evaluating at $K$ and using that $G^\tt(K)=G(K)^\tt$ by \cref{l:hat-vs-p-divisible-groups}.1 and 3.
\end{proof}
\begin{Remark}
	\begin{enumerate}
	\item If $K=\C_p$ and the identity component $G^0$ is abeloid, then by Lemma~\ref{l:tt-over-C_p}, $G^\tt(\C_p)=G^\tau(\C_p)$ is generated by $G^0(\C_p)$ and torsion in the N\'eron--Severi group. In particular, these are then precisely the line bundles in $\Pic(X)$ with vanishing Chern classes. However, this is no longer true over extensions of $\C_p$.
	\item Even without assuming any representability results, $\uP_{X,\et}$ is always a small \mbox{$v$-sheaf} (this follows from \cite[Theorem~7.13]{heuer-sheafified-paCS}), hence $\uP_{X,\et}(K)$  always has the natural structure of a topological space by \cite[Definition 12.8]{etale-cohomology-of-diamonds}. One can then still describe $\uP_{X,\et}^{\tt}(K)$ as those $x$ in $\uP_{X,\et}(K)$ such that $x^{n!}\to 1$ for $n\to \infty$.
	\item 	The sequence in is not in general right-exact: If $X$ is an abelian variety with good supersingular reduction $\overline{X}$, then $\Pic(\wt X)=\Pic(\overline{X})\otimes \Q$ by \cite[Theorem~4.1]{heuer-Picard-good-reduction}.
	\item 	We cannot expect $\uP_{\wt X,\et}$ to be represented by an adic group: If $\uP_{X,\et}$  and $\uP_{\wt X,\et}$ were both representable by adic groups, then the kernel of $\uP_{X,\et}\to \uP_{\wt X,\et}$ would be closed, a contradiction to \cref{l:hat-vs-p-divisible-groups} if $\uP^0_{X,\et}$ is a rigid group of dimension $>0$. In fact, we show in \cite[\S 6.5]{heuer-Picard-good-reduction} that $\uP_{\wt X,\et}$ is typically not even a diamond.
	\end{enumerate}
\end{Remark}

\subsection{Relative universal property of the universal cover}
The first step for the proof of Theorem~\ref{t:tt-torsion-is-1+m-torsors} is to show that $\G_m^{\tt}$-torsors become trivial on the pro-finite-\'etale universal cover $\wt X\to X$.
This is part of the following analogue of the statement that the complex universal cover is simply connected. We recall that we write $\wt \pi:\wt X\to \Spa(K)$ for the structure map.

\begin{Proposition}\label{p:cohomology-of-wtXxY}
	Let $\mathcal F$ be any one of the $v$-sheaves $\O^{+a}/p^k$, $\O^{+a}$, $\O$, $\O^{\flat+ a}$, $\O^\flat$, $\G_m\tpt$, $\G_m^{\tt}$, or $\Z/N\Z$ for $N\in \N$.
		Then for any affinoid perfectoid space $Y$ over $K$, we have 
		\[ \begin{tikzcd}[row sep =0.15cm, column sep =0.15cm]
			H^n_{\et}(\wt X\times Y,\mathcal F)=H^n_{\et}(Y,\mathcal F)=H^n_{v}(Y,\mathcal F)=H^n_{v}(\wt X\times Y,\mathcal F).
		\end{tikzcd}\]
	for $n\in \{0,1\}$.
		 In particular, $\pi_{\et\ast}\mathcal F=\mathcal F=\pi_{v\ast}\mathcal F$ and  $R^1\wt \pi_{\et\ast}\mathcal F=0=R^1\wt \pi_{v\ast}\mathcal F$ on $\Perf_K$.
		\end{Proposition}
	\begin{Remark}
		The restriction to $n\in \{0,1\}$ is necessary in general:
		The Proposition does not hold for $n=2$ for any of the sheaves in the example of $X=\P^1=\wt X$ and $Y=\Spa(K)$.
	\end{Remark}
	\begin{proof}
		This is a relative version of \cite[Proposition~3.10]{heuer-v_lb_rigid}, and the technical results of \cite[\S4]{heuer-diamantine-Picard} enable us to essentially follow the same line of argument:
		We start with $\mathcal F=\O^+/p^k$. By \cite[Proposition~4.12.2]{heuer-diamantine-Picard}, which we recall below as \cref{l:diam-Prop-4.12.2},
		\[ H^1_v(\wt X\times Y,\O^+/p^k)\aeq \varinjlim_{X'\to X}H^1_{\et}(X'\times Y,\O^+/p^k),\]
		where the $X'$ are as in the definition of $\wt X$.
		By \cite[Corollary~4.6]{heuer-diamantine-Picard}, this equals
		\[
		 ...\aeq \varinjlim_{X'\to X} H^1_\et(X',\Z/p^k)\otimes \O^+(Y)/p^k=0\]
		which vanishes because every class in $H^1_\et(X',\Z/p^k)$ is trivialised by a connected finite \'etale cover of $X'$ and thus of $X$. This gives the case of $n=1$. 
		
		The case $n=0$ follows from \cref{l:diam-Prop-4.12.2} and \cite[Proposition~4.2.2]{heuer-diamantine-Picard} which says that
		\[H^0(X'\times Y,\O^+/p^k)\aeq H^0(Y,\O^+/p^k).\]
		
		Next, consider $\mathcal F=\mathcal O^+$: The last equation implies that 
		$R^1\varprojlim_{k\in \N} H^0(\wt X\times Y,\O^+/p^k)\aeq 0$, so we deduce from the fact that the $v$-site is replete and \cite[Proposition 3.1.10]{bhatt-scholze-proetale} that
		\[H^1_v(\wt X\times Y,\O^+)\aeq \varprojlim_{k\in \N} H^1_v(\wt X\times Y,\O^+/p^k)\aeq 0.\]
		The case of $\O$ follows. A similar limit argument gives the cases of $\O^{\flat+a}$ and then $\O^\flat$. 
		Since we have an injection $H^1_{\et}(\wt X\times Y,\mathcal F)\hookrightarrow H^1_{v}(\wt X\times Y,\mathcal F)$, the \'etale case follows for all of these.
		
		For $\mathcal F=\Z/N\Z$, we know that \'etale and $v$-topology agree. By \cite[Proposition~14.9]{etale-cohomology-of-diamonds},
		\[ H^n_{\et}(\wt X\times Y,\Z/N\Z)=\varinjlim H^n_{\et}(X'\times Y,\Z/N\Z).\]
		For $n=0$, it follows from  \cite[Corollary~4.8]{heuer-diamantine-Picard} that $\wt\pi_{\ast}\Z/N\Z=\Z/N\Z$. To see the case of $n=1$, it thus suffices by the Leray sequence to prove that $R^1\wt\pi_{\et\ast}\Z/N\Z=0$, which follows from the same cited result in the colimit over the $X'$.
		
		Finally, the case of $\G_m\tpt$ follows from that of $\mu_{p^\infty}\cong \varinjlim \Z/p^k\Z$ (as $K$ is algebraically closed) and that of $\O$ by the logarithm sequence, and similarly for $\G_m^\tt$ where we also include coprime-to-$p$ torsion. This also shows that \'etale and $v$-cohomology agree in this case.
	\end{proof}

\begin{Lemma}\label{l:diam-Prop-4.12.2}
	Let $Y$ be any perfectoid space over $K$ and let $\wt X=\varprojlim_{i\in I} X_i$ be a diamond which is a limit of smooth qcqs rigid spaces over $K$ with finite \'etale transition maps. Then for the $v$-sheaves $\mathcal F=\O^+/p$ or $\mathcal F=\G_m/\G_m^{\tt}$, we have for $n\in\{0,1\}$:
	\[H^n_{v}(\wt X\times Y,\mathcal F)=\textstyle\varinjlim_{i \in I} H^n_{\et}(X_i\times Y,\mathcal F).\]
	Moreover, both sides stay the same when we exchange the $v$-topology for the \'etale topology.
\end{Lemma}
\begin{proof}
	For $\mathcal F=\O^+/p$ and $\mathcal F=\G_m/\G_m\tpt$ this is \cite[Proposition~4.12.1 and 2]{heuer-diamantine-Picard} (where the latter sheaf was denoted by $\bOx$). The case of $\G_m/\G_m^{\tt}$ follows by tensoring with $\Q$ since $\G_m/\G_m^{\tt}=(\G_m/\G_m\tpt)\otimes_{\Z} \Q$, see \cite[Lemma 2.16]{heuer-v_lb_rigid}.
\end{proof}
\begin{Corollary}\label{l:uP_wtX-et-vs-v}
	The natural map $\uP_{\wt X,\et}\to \uP_{\wt X,v}$ is injective.
\end{Corollary}
\begin{proof}
	By Proposition~\ref{p:cohomology-of-wtXxY}, we have $\wt\pi_{\ast}\O=\O$ and thus $\wt\pi_{\ast}\G_m=\G_m$. We can therefore apply \cite[Lemma~4.13]{heuer-diamantine-Picard}, which gives a commutative diagram 
	with short exact rows:
	\[ \begin{tikzcd}
		1 \arrow[r] & {H^1_{\et}(Y,\G_m)} \arrow[r] \arrow[d,equal] & {H^1_{\et}(\wt X\times Y,\G_m)} \arrow[r] \arrow[r] \arrow[d] & {\uP_{\wt X,\et}(Y)} \arrow[r] \arrow[d] & 1 \\
		1 \arrow[r] & {H^1_{v}(Y,\G_m)} \arrow[r] & {H^1_{v}(\wt X\times Y,\G_m)} \arrow[r] & {\uP_{\wt X,v}(Y)} \arrow[r] & 1
	\end{tikzcd}\] As the first column is an isomorphism by \cite[Theorem~3.5.8]{KedlayaLiu-II}, and the middle column is clearly injective, we can deduce that the last one is injective as well.
\end{proof}
\begin{Remark}
	We do not know whether the map in \cref{l:uP_wtX-et-vs-v} is an isomorphism in general. While this is true for curves,  it is certainly not always true that $v$-vector bundles on $\wt X$ agree with \'etale vector bundles, as the example of $X=\wt X=\P^1$ shows.
\end{Remark}
\begin{Remark}
	An alternative proof of \cref{p:cohomology-of-wtXxY} in the case of  $\Z/N\Z$ and $n=1$ would be to show that 
	$(\wt X\times Y)_{\fet}=Y_{\fet}$.
	One could deduce this from the following rigid analogue of a well-known algebraic statement:
	Let $X$ be a connected proper rigid space and $Y$ any connected reduced rigid space. Then is the natural map $\pi^{\et}_1(X\times Y)\to \pi_1^{\et}(X)\times \pi^{\et}_1(Y)$
		an isomorphism?
	We suspect that one could see this like in \cite[X Corollaire 1.7]{SGA1}, or using \cite[\S 16]{ScholzeBerkeleyLectureNotes}. By an approximation argument, one could deduce that  $(\wt X\times Y)_{\fet}=Y_{\fet}$.
\end{Remark}
\subsection{Topological torsion line bundles}
Next, we show that any pro-finite-\'etale line bundle admits a reduction of structure group from $\G_m$ to $\G_m^\tt$. For this we first show:
\begin{Lemma}\label{l:H^1(X,O/O^tt)->H^1(wtX,O/O^tt)-is-isom}
	For any perfectoid space $Y$,  there is for $n\in \{0,1\}$ a natural isomorphism 
	\[ H^n_{v}(X\times Y,\G_m/\G_m^\tt)\isomarrow H^n_{v}(\wt X\times Y,\G_m/\G_m^\tt)^{\pi_1(X)}.\]
		Moreover, both sides stay the same when we exchange the $v$-topology for the \'etale topology.
\end{Lemma}
\begin{proof}
	We form the Cartan--Leray spectral sequence of the $\pi_1(X)$-torsor $\wt X\times Y\to X\times Y$ for the sheaf $\G_m/\G_m^\tt$, see \cite[Proposition~2.8]{heuer-v_lb_rigid}: Endow $H^n_v(\wt X\times Y,\G_m/\G_m^\tt)$ with the discrete topology, then by \cref{l:diam-Prop-4.12.2}, we have for any profinite group $G$:
	\[ H^n_v(G\times \wt X\times Y,\G_m/\G_m^\tt)=\Map_{\cts}(G,H^n_v(\wt X\times Y,\G_m/\G_m^\tt)).\]
	For $n=0$, the lemma follows by setting $G=\pi_1(X)$ and using the $v$-sheaf property of $\G_m/\G_m^\tt$. For $n=1$, this equation ensures that the conditions of \cite[Proposition~2.8.2]{heuer-v_lb_rigid} are satisfied. Setting $G:=\pi_1(X)$ and $\mathcal F=\G_m/\G_m^\tt$, we thus get an exact sequence
	\[0\to H^1_{\cts}(G,\mathcal F(\wt X\times Y))\to H^1_v(X\times Y,\mathcal F)\to H^1_v(\wt X\times Y,\mathcal F)^{G}\to H^2_{\cts}(G,\mathcal F(\wt X\times Y)). \]
	
	Since $G=\pi_1(X)$ is profinite and $\mathcal F=\G_m/\G_m^\tt$ is a sheaf of $\Q$-vector spaces, the outer two groups vanish by \cite[Proposition 1.6.2.c]{NeuSchWin}. This gives the desired statement for the $v$-topology.
	The \'etale case follows from this by \cref{l:diam-Prop-4.12.2} and \cite[Proposition 4.12.1]{heuer-diamantine-Picard}, or alternatively from the more general \cite[Proposition~4.8]{heuer-G-torsors-perfectoid-spaces}.
\end{proof}

\begin{Lemma}[{\cite[Lemma~4.11.2]{heuer-diamantine-Picard}}]\label{l:I-4-11.2}
	We have $\pi_{\tau\ast}(\G_m/\G_m^\tt)=\G_m/\G_m^\tt$ for $\tau\in \{\et,v\}$.
\end{Lemma}
\begin{proof}
	The reference shows this for $\G_m/\G_m\tpt$, the lemma follows by tensoring with $\Q$.
\end{proof}

We can now complete the first steps of the proof of Theorem~\ref{t:tt-torsion-is-1+m-torsors}:
\begin{proof}[Proof of Theorem~\ref{t:tt-torsion-is-1+m-torsors}]
	Let $\tau\in \{\et,v\}$
	 and consider the morphism of exact sequences
	\[
\begin{tikzcd}
	0\arrow[r,equal] &{R^1\wt\pi_{\tau\ast}\G_m^\tt} \arrow[r] & R^1\wt\pi_{\tau\ast}\G_m \arrow[r] & R^1\wt\pi_{\tau\ast}(\G_m/\G_m^\tt) \\
	0\arrow[r] &{R^1\pi_{\tau\ast}\G_m^{\tt}} \arrow[r] \arrow[u] & {R^1\pi_{\tau\ast}\G_m } \arrow[r] \arrow[u] & {R^1\pi_{\tau\ast}(\G_m/\G_m^\tt).} \arrow[u,hook]
\end{tikzcd}
\]
Here the top left entry vanishes by Proposition~\ref{p:cohomology-of-wtXxY}. The right vertical morphism is injective by Lemma~\ref{l:H^1(X,O/O^tt)->H^1(wtX,O/O^tt)-is-isom}. The bottom left morphism is injective  by \cref{l:I-4-11.2}

This shows that  we have a left-exact sequence:
\begin{equation}\label{eq:R^1pi_tt-is-kernel-of-uPic_X->uPic_wtX}
0\to R^1\pi_{\tau\ast}\G_m^{\tt}\to \uP_{X,\tau} \to \uP_{\wt X,\tau}
\end{equation}
For Theorem~\ref{t:tt-torsion-is-1+m-torsors}.1, it remains to prove that
$R^1\pi_{\tau\ast}\G_m^{\tt}=\uP^\tt_{X,\tau}$. This relies on the following key calculation, which crucially uses that $\uP_{X,\tau}$ is defined on perfectoid test objects:

\begin{Proposition}\label{p:Pic-wt-X-is-top-tf}
	We have $\HOM(\whZ,\uP_{\wt X,\et})=1$ and thus 
	$\uP_{\wt X,\et}^{\tt}=1$. 
\end{Proposition}
	\begin{proof}
		The first statement implies the second by the definition of $-^\tt$ in Definition~\ref{d:tt}.
		
		For any affinoid perfectoid $Y$ over $K$, the internal Hom sheaf is given by
		\begin{align*}
		\HOM(\whZ,\uP_{\wt X,\et})(Y)&=\eq \big(\uP_{\wt X,\et}(\whZ\times Y)\rightrightarrows \uP_{\wt X,\et}(\whZ^2\times Y)\big)\\
		&=\eq \big(\uP_{\whZ\times\wt X,\et}( Y)\rightrightarrows \uP_{\whZ^2\times \wt X,\et}(Y)\big),
		\end{align*}
		where the last equality holds by Lemma~\ref{l:pro-finite-sets-in-uPic} below. The last term is the sheafification of
		\[ Y\mapsto \eq \big(\Pic_{\et}(\whZ\times \wt X\times Y)\rightrightarrows \Pic_{\et}(\whZ^2\times \wt X\times Y)\big)\]
		on $\Perf_{K,\et}$.
		We claim that this sheaf vanishes:
	For this we use that by Proposition~\ref{p:cohomology-of-wtXxY}, 
	$\wt \pi_{\ast}\G_m^{\tt}=\Ottt$ and $R^1\wt \pi_{\ast}\Ottt=0$, which by Lemma~\ref{l:pro-finite-sets-in-uPic} implies $R^1q_{\ast}\Ottt=0$ for the structure map $q:\whZ^k\times \wt X\to \Spa(K)$ for any $k\in \N$. It therefore suffices to prove that the presheaf
	\[ Y\mapsto \mathrm{eq}\big(H^1_{\et}(\wh{\Z}\times \wt X\times Y,\bOttt)\rightrightarrows H^1_{\et}(\wh{\Z}^2\times \wt X\times Y,\bOttt)\big)\]
	on $\Perf_{K}$
	is trivial. For this, we use that by \cref{l:diam-Prop-4.12.2},
	we have for any $k\in \N$
	\[H^1_{\et}(\wh{\Z}^k\times \wt X\times Y,\bOttt)=\Map_{\lc}(\wh{\Z}^k,H^1_{\et}(\wt X\times Y,\bOttt)),\]
	which shows that this equaliser is precisely
	\[\Hom_{\lc}(\wh{\Z},H^1_{\et}(\wt X\times Y,\bOttt)).\]
	But the second argument is a $\Q$-vector space, so this group is trivial, as desired.
\end{proof}
\begin{Lemma}\label{l:pro-finite-sets-in-uPic}
	For any spatial diamonds $X$ and $Y$ over $K$ and any profinite set $S$, write
	\[X\times Y\times S\xrightarrow{\pi_1} Y\times S\xrightarrow{\pi_2} Y\]
	for the projections.
	Then for any abelian sheaf $F$ on $\LSD_{K,\et}$ and any $n\geq 0$, we have
	\[R^n(\pi_2\circ\pi_1)_{\et\ast}F=\pi_{2,\et\ast}R^n\pi_{1,\et\ast}F.\]
\end{Lemma}
\begin{proof}
	Write $S=\varprojlim S_i$, then  $Y\times S=\varprojlim Y\times S_i$. Hence by \cite[Proposition~11.23]{etale-cohomology-of-diamonds}, any \'etale cover of $Y\times S$ can be refined by one of the form $\mathfrak V\times\mathfrak U$  where $\mathfrak U$ is a disjoint analytic cover of $S$ and $\mathfrak V$ is an \'etale cover of $Y$. Therefore $R\pi_{2,\et\ast}G=\pi_{2,\et\ast}G$ for any abelian sheaf $G$ on $(Y\times S)_{\et}$. The desired statement follows from the Grothendieck spectral sequence.
\end{proof}

To prove Theorem~\ref{t:tt-torsion-is-1+m-torsors}, we now apply the left-exact functor $\HOM(\whZ,-)$ to \eqref{eq:R^1pi_tt-is-kernel-of-uPic_X->uPic_wtX} and find
\[\HOM(\whZ,R^1\pi_{\et\ast}\Ottt)=\HOM(\whZ,\uP_{X,\et})\]
by Proposition~\ref{p:Pic-wt-X-is-top-tf}, hence
\[(R^1\pi_{\et\ast}\Ottt)^\tt=\uP_{X,\et}^{\tt}.\]
Since $\Ottt$ is strongly topologically torsion by \cref{l:hat-vs-p-divisible-groups}, Lemma~\ref{l:cohomology-of-stt-is-tt} guarantees that $R^1\pi_{\et\ast}\Ottt$ is topologically torsion, thus the left hand side equals 
\[(R^1\pi_{\et\ast}\Ottt)^\tt=R^1\pi_{\et\ast}\Ottt.\]
This finishes the proof of part 1 of Theorem~\ref{t:tt-torsion-is-1+m-torsors}.

For the second part, we need to incorporate the character variety into the picture.
\subsection{Relation to the character variety}
The Cartan--Leray short exact sequence for the $\pi_1(X)$-torsor $\wt X\times Y\to X\times Y$ and the sheaf $\G_m$ yields an exact sequence
\begin{equation}\label{eq:CL-for-G_m}
0\to H^1_{\cts}(\pi_1(X),\G_m(\wt X\times Y))\to \Pic_v(X\times Y)\to \Pic_v(\wt X\times Y).
\end{equation}
By Proposition~\ref{p:cohomology-of-wtXxY}, $\G_m(\wt X\times Y)=\G_m(Y)$, so the first term equals $\Hom_{\cts}(\pi_1(X),\G_m(Y))$.
Using the character variety from \cref{l:cts-char-variety}, we thus obtain a left-exact sequence on $\Perf_{K,v}$
\[0\to\HOM(\pi_1(X),\G_m)\to  \uP_{X,v}\to  \uP_{\wt X,v}\]
We can describe the composition of the first map  with $\HTlog$ more explicitly:
Recall that  $\HOM(\pi_1(X),\G_m)=\HOM(\pi_1(X),\G^\tt_m)$, so we can compose characters with $\log:\G_m^\tt\to \G_a$.
\begin{Lemma}\label{l:explicit-descr-of-HTlog-on-HOM}
	The following square is commutative and has a surjective diagonal:
	\[\begin{tikzcd}
			{\HOM(\pi_1(X),\G_m)}\arrow[r] \arrow[rd,dashed]\arrow[d,"\log"] & {\uP_{X,v}} \arrow[d,"\HTlog"]\\
	{\HOM(\pi_1(X),\G_a)}\arrow[r,"\HT"]& {H^0(X,\Omega^1_X(-1))\otimes \G_a}&
\end{tikzcd}\]
\end{Lemma}
\begin{proof} By functoriality of the Cartan--Leray  sequence (left square) and the definition of $\HTlog$ in \cite[Proposition~2.15]{heuer-diamantine-Picard} (right square), the following  diagram commutes:
\[
\begin{tikzcd}
	0 \arrow[r] & {\Hom_{\cts}(\pi_1(X),\O^\times(Y))} \arrow[r] \arrow[d,"\log"] & {H^1_v(X\times Y,\G_m^\tt)} \arrow[d,"\log"] \arrow[r,"\HTlog"] & {H^0(X,\Omega^1_X(-1))\otimes \O(Y)} \arrow[d,equal] \\
	0 \arrow[r] & {\Hom_{\cts}(\pi_1(X),\O(Y))} \arrow[r,"\sim"]                  & {H^1_v(X\times Y,\O)} \arrow[r,"\HT"]                & {H^0(X,\Omega^1_X(-1))\otimes \O(Y)}          
\end{tikzcd}\]
This shows commutativity.
 That the left map is surjective follows from \cref{l:cts-char-variety} and surjectivity of $\log:\G_m^\tt\to \G_a$. The map $\HT$ is surjective by \cite[Proposition~2.6]{heuer-diamantine-Picard}.
\end{proof}

We can now put everything together: The diagram from \cref{l:explicit-descr-of-HTlog-on-HOM} combines with \cref{eq:CL-for-G_m} and the sequences from  Theorem~\ref{t:representability-of-v-Picard-functor} and Theorem~\ref{t:tt-torsion-is-1+m-torsors}.1 to a commutative diagram 
\[\begin{tikzcd}
		0\arrow[r]&R^1\pi_{\et\ast}\Ottt \arrow[d,dotted] \arrow[r] & {\uP_{X,\et}} \arrow[d] \arrow[r] & {\uP_{\wt X,\et}} \arrow[d,hook]\\
		0\arrow[r]&{\HOM(\pi_1(X),\G_m)}\arrow[r] \arrow[d,"\log"] \arrow[rd,dashed] & {\uP_{X,v}} \arrow[r] \arrow[d]\arrow[ur,dashed] & {\uP_{\wt X,v}}\\
		&{\HOM(\pi_1(X),\G_a)}\arrow[r,"\HT"]& {H^0(X,\Omega^1_X(-1))\otimes \G_a}&
\end{tikzcd}\]
with left-exact rows. 
By \cref{l:uP_wtX-et-vs-v}, the morphism in the rightmost column is injective. By \cref{l:explicit-descr-of-HTlog-on-HOM}, the bottom dashed diagonal map is surjective. 
 It follows from a diagram chase that  \mbox{$\uP_{X,v}\to \uP_{\wt X,v}$} admits a factorisation through the top dashed arrow.
Applying $-^\tt$, we deduce from \cref{p:Pic-wt-X-is-top-tf} that this map becomes trivial, and we can thus argue as in the first part to deduce from the middle exact sequence  that
\[ \HOM(\pi_1(X),\G_m)= \HOM(\pi_1(X),\G_m)^\tt=\uP_{\wt X,v}^\tt,\]
where the first equality follows from \cref{l:cts-char-variety}. This finishes the proof of Theorem~\ref{t:tt-torsion-is-1+m-torsors}.2.
\end{proof}

\section{The morphism of Deninger--Werner}
We now use Theorem~\ref{t:tt-torsion-is-1+m-torsors} to  give a reinterpretation, a generalisation, and a geometrisation of a construction of Deninger--Werner in the context of the $p$-adic Simpson correspondence.

Let $X$ be a connected smooth projective curve over $\overline{\Q}_p$. We fix a base-point $x\in X(\overline{\Q}_p)$.
In \cite{DeningerWerner_vb_p-adic_curves}, Deninger--Werner construct a functor from a certain category of vector bundles on $X_{\C_p}$ to $\C_p$-linear representations of the \'etale fundamental group $\pi_1(X):=\pi_1^{\et}(X,x)$ of $X$, thus defining a $p$-adic Simpson functor in the case of vanishing Higgs field, or in other words a partial analogue of Narasimhan--Seshadri theory. In \cite{DeningerWerner-lb_and_p-adic-characters}, they go on to study this functor in the case of line bundles under the additional assumption that $X$ has good reduction: In this case, their functor induces an injective continuous homomorphism
\[ \alpha:\uP^{0}_X(\C_p)\to \Hom_{\cts}(\pi_1(X),\C_p^\times)=\Hom_{\cts}(TA^\vee,\C_p^\times),\]
where $A=\uP^{0}_X(\C_p)$ is the Jacobian of $X$ and $TA^\vee$ is the adelic Tate module of its dual, i.e.\ of the Albanese. Deninger--Werner give an explicit description of $\alpha$ in terms of the Weil pairing of $A$, and extend their construction to any connected smooth proper algebraic variety $X$ over $\overline{\Q}_p$ satisfying a certain good reduction assumption (see \cite[\S1.5]{DeningerWerner_vb_p-adic_curves}), obtaining more generally a morphism  $\alpha$ defined  on the $\C_p$-points of the open and closed subgroup  $\uP^{\tau}_X$ of the Picard variety whose image in the N\'eron--Severi group is torsion.

In the case that $X$ is a curve over $\overline{\Q}_p$ with good reduction, Song \cite{song2020rigid} has recently shown that the morphism $\alpha$ can be ``geometrised'', i.e.\ interpreted as the $\C_p$-points of a morphism of rigid group varieties.

\subsection{Geometrisation of Deninger--Werner's morphism}

We now generalise and geometrise the results of Deninger--Werner and Song to any smooth proper rigid space $X$ of arbitrary dimension over any complete algebraically closed field $K$ over $\Q_p$. For this we use a completely different method based on  \cref{t:representability-of-v-Picard-functor} and \cref{t:tt-torsion-is-1+m-torsors}.
Namely, have the following result, which includes a topological torsion version of \cref{t:representability-of-v-Picard-functor}:

\begin{Theorem}\label{t:geometric-Simpson}
	Let $X$ be a connected smooth proper rigid space over $K$. Then the topological torsion Picard functor
	\[\uP_{X,\et}^{\tt}=R^1\pi_{\et\ast}\Ottt\]
	is always represented by a disjoint union of divisible rigid analytic groups. It fits into a natural short exact sequence of rigid group varieties
	\begin{equation}\label{seq:Picv-tt-ses}
	0\to \uP_{X,\et}^{\tt}\to \HOM(\pi_1(X),\G_m)\xrightarrow{\HTlog} H^0(X,\Omega^1_X(-1))\otimes \G_a\to 0
	\end{equation}
 that is functorial in $X\to \Spa(K)$.
It is canonically isomorphic to the topological torsion part of the exact sequence \cref{seq:Picv-ses} in  \cref{t:representability-of-v-Picard-functor}. The induced exact sequence  on Lie algebras, i.e.\ tangent spaces at the identity, is canonically identified with the Hodge--Tate sequence
\[0\to H^1_{\an}(X,\O)\to H^1_{\et}(X,\Q_p)\otimes_{\Q_p} K\xrightarrow{\HT} H^0(X,\Omega^1(-1))\to 0.\]
The first map of \cref{seq:Picv-tt-ses} induces a natural non-degenerate pairing of adic groups
	\[ \uP_{X,\et}^{\tt}\times \pi_1(X)\to \G_m.\]
	If $\uP_{X,\et}^{\tau}$ is an abeloid variety, this is the unique analytic continuation of the Weil pairing.
\end{Theorem}
The first morphism in \cref{{seq:Picv-tt-ses}} generalises the map $\alpha$ of Deninger--Werner and Song. Indeed, if $\uP^{\tau}_{X,\et}$ is abeloid and $K=\C_p$, then by Lemma~\ref{l:tt-over-C_p}, we have $\uP_{X,\et}^{\tt}(\C_p)=\uP_{X,\et}^{0}(\C_p)$.

But this is no longer true over more general fields, where we need to use the topological torsion subsheaf rather than all of $\uP_{X,\et}^\tau$ to define $\alpha$.
The last part is our analogue of the statement in \cite[\S4]{DeningerWerner-lb_and_p-adic-characters} that $\alpha$ is an extension of the Weil pairing in the case of curves.

\begin{Remark}
	We note that in contrast to \cite[Theorem~1.0.1]{song2020rigid}, the domain of our morphism is not $\uP_{X,\et}$ but the open subspace $\uP_{X,\et}^{\tt}\subseteq \uP_{X,\et}$.  For curves over $\C_p$, both spaces have the same $\C_p$-points, but the latter has more points when considered as an adic space. 
	In particular,  the assumptions of \cite[Theorems~1.0.1 and 1.0.3]{song2020rigid} are never satisfied. Indeed, note that any morphism from the proper space $\uP^0_{X,\et}$ to $\G_m$ is trivial. That said, for curves, one can use Song's construction to explicitly define a map on $\uP_{X,\et}^{\tt}$.
\end{Remark}

\begin{Remark}\label{rm:complex-Simpson}
	As pointed out in \cref{r:Remark-comparison-to-complex}, in the setting of curves,
	the analogue of the short exact sequence of Theorem~\ref{t:geometric-Simpson} in the complex analytic setting of the Corlette--Simpson correspondence
is only real-analytic. This difference to the $p$-adic case can in part be attributed to the fact that the pro-\'etale fundamental group appearing in Theorem~\ref{t:geometric-Simpson} is profinite and thus ignores the second factor in the decomposition $\C_p^\times=\O_{\C_p}^\times\times \Q$.
\end{Remark}

\begin{proof}[Proof of \cref{t:geometric-Simpson}]
	The exact sequence arises from the one in Theorem~\ref{t:representability-of-v-Picard-functor} by applying the functor  $-^\tt$ from \cref{s:tt-sheaves}. To see that this preserves exactness, we use that by \cite[Theorem~2.7.3]{heuer-diamantine-Picard}, this sequence is split over the open subgroup 
	\[\mathcal A^{+}\subseteq \mathcal A:=H^0(X,\Omega^1)(-1)\otimes_K \G_a\]
	 defined by the image of $\HOM(\pi_1(X),p\G_a^+)$ under the morphism $\HT:\HOM(\pi_1(X),\G_a)\to \mathcal A$.
	More explicitly, let $\uP_{X,v}^+$ be the pullback under $\HTlog$ of $ \mathcal A^+$, then the induced sequence
	\[0\to \uP_{X,\et}\to \uP_{X,v}^+\to \mathcal A^+\to 0\]
	is split, and in particular exactness is preserved by the additive functor $-^{\tt}$. This shows left-exactness. 
	Right-exactness follows from \cref{l:explicit-descr-of-HTlog-on-HOM}.
	
	In order to deduce the description of the induced map on tangent spaces, we note that adding kernels to \cref{l:explicit-descr-of-HTlog-on-HOM} induces a morphism of short exact sequences:
		\[\begin{tikzcd}
		0\to {\uP_{X,\et}^{\tt}}\arrow[r]\arrow[d,dotted,"\log"]&{\HOM(\pi_1(X),\G_m)}\arrow[r,"\HTlog"] \arrow[d,"\log"] & {H^0(X,\Omega^1_X(-1))\otimes \G_a} \arrow[d,equal]\arrow[r]&0\\
		0\to {H^1_{\an}(X,\O)\otimes \G_a}\arrow[r]&{\HOM(\pi_1(X),\G_a)}\arrow[r,"\HT"]& {H^0(X,\Omega^1_X(-1))\otimes \G_a}\arrow[r]&0
	\end{tikzcd}\]
	Since the middle arrow is an isomorphism in a neighbourhood of the identity (because $\log$ is), this induces an isomorphism between the associated sequences on tangent spaces.

	Finally, the pairing is tautologically associated to the map  $ \uP_{X,\et}^{\tt}\to \HOM(\pi_1(X),\G_m)$.
	It remains to explain why this pairing is an analytic continuation of the Weil pairing if $\uP_{X,\et}^{0}$ is abeloid. Via the theory of the rigid Albanese variety developed in \cite[\S4]{HansenLi_HodgeSymmetry}, and using the functoriality of Theorem~\ref{t:representability-of-v-Picard-functor}, we can reduce to the case of abeloid $X$. Describing the pairing in terms of the Weil pairing in the abeloid case is the goal of the next subsection. 
	
	Before, we give two Corollaries. By comparing to \cref{t:representability-of-v-Picard-functor}, we immediately see:
	\begin{Corollary}\label{c:pushout-relation}The map
		$\HTlog:\uP_{X,v}\to \mathcal A$ is the \'etale $\uP_{X,\et}$-torsor given by pushout of the  $\uP_{X,\et}^{\tt}$-torsor $\HOM(\pi_1(X),\G_m)\to \mathcal A$ from \cref{seq:Picv-tt-ses} along the morphism $\uP_{X,\et}^{\tt}\to\uP_{X,\et}$. 
	\end{Corollary}
	
	\begin{Corollary}\label{c:torsors-not-split}
		The sequences \eqref{seq:Picv-ses} and \eqref{seq:Picv-tt-ses}  are never split unless $H^0(X,\Omega_X^1)=0$.
	\end{Corollary}
	\begin{proof}
		Any splitting of \eqref{seq:Picv-tt-ses} would be a morphism $\mathcal A\to \HOM(\pi,\G_m)$ from an affine rigid space to a bounded subspace of $\G_m^d$ for some $d\in \N$. But any such morphism is trivial.
		
		The statement for  \eqref{seq:Picv-ses}  follows as we obtain  \eqref{seq:Picv-tt-ses} from \eqref{seq:Picv-ses}  by applying the functor $-^\tt$.
	\end{proof}
	
\subsection{The Weil pairing from a pro-\'etale perspective}
It remains to explain why the pairing from Theorem~\ref{t:geometric-Simpson} can be described as a unique analytic continuation of the Weil pairing if $X$ is abeloid. This is a pleasant application of Theorem~\ref{t:tt-torsion-is-1+m-torsors} in its own right.
\begin{Remark}
We note that it has long been known that such analytic Weil pairings exists: For abelian varieties of good reduction this is constructed in \cite[\S4]{tate1967p}. Another related construction is by Fontaine \cite[Proposition~1.1]{FontainePresqueCp} who gives for any abelian variety $A$ over a  finite extension $L$ of $\Q_p$ a natural continuous Galois-equivariant homomorphism that reinterpreted in our notation (via Lemma~\ref{l:tt-over-C_p} and Proposition~\ref{l:hat-vs-p-divisible-groups}) is of the form
\[A\tpt(\overline L)\to T_pA(-1)\otimes (1+\mathfrak m_{\C_p})=\HOM(T_pA^\vee,\G_m).\]
By comparing Fontaine's construction to ours, it is clear that (the topological $p$-torsion part of) the analytic Weil pairing from \cref{t:geometric-Simpson} is a geometrisation of Fontaine's map.
\end{Remark}

Let $A$ be an abeloid variety over $K$. The Weil pairing of $A$ is the perfect pairing 
\[\mathbf e_N:A[N]\times A^\vee[N]\to \mu_{N}\] defined as follows: Let $\mathcal P$ be the Poincar\'e bundle on $A\times A^\vee$, then by bilinearity  its restriction to $A\times A^\vee[N]$ has trivial $N$-th tensor power. It follows that $\mathcal P$ admits a canonical reduction of structure group to a $\mu_N$-torsor $\mathcal P_N$ on $A\times A^\vee$ that becomes trivial after pullback along $[N]:A\to A$. Consequently, we have a natural commutative diagram of bi-extensions
\begin{equation}\label{eq:Poincare-bundle-vs-finite-Weil-pairing}
	\begin{tikzcd}
		0 \arrow[r] &A[N]\times A^\vee[N]\arrow[r]\arrow[d,"\mathbf e_N"]& A\times A^\vee[N] \arrow[d] \arrow[r,"{[N]\times \id}"] & A\times A^\vee[N] \arrow[d, equal] \arrow[r] & 0 \\
		0 \arrow[r] &\mu_{N} \arrow[r] & \mathcal P_N \arrow[r] & A\times A^\vee[N]\arrow[r] & 0
	\end{tikzcd}
\end{equation}
and the Weil pairing is defined as the induced arrow on the left.

We can now extend this construction using the universal cover $\wt A$ of $A$ from \cref{s:profet-cover}. In the abeloid case, this is given by the space $\wt A=\textstyle\varprojlim_{[N],N\in\N}A$ and thus sits in an exact sequence
\[ 0\to TA\to \wt A\to A\to 0\]
where $TA=\pi^\et_1(A,0)=\varprojlim_{N\in \N} A[N]$ is the Tate module of $A$. We now restrict to the open subspace $A\times A^{\vee\tt}\subseteq A\times A^\vee$. By Theorem~\ref{t:tt-torsion-is-1+m-torsors}.1, this is the locus where $\mathcal P$ admits a reduction of structure group to a $\Ottt$-torsor $\wh{\mathcal P}$, which inherits from $\mathcal P$ the structure of a bi-extension
\[ 1\to \G_m^\tt\to \wh{\mathcal P}\to A\times A^{\vee\tt}\to 0.\]
Concretely, this means that it becomes an exact sequence of abelian sheaves on $\mathrm{LSD}_{A}$ and $\mathrm{LSD}_{A^{\vee\tt}}$ respectively when restricted to test objects relatively over either factor.

By Theorem~\ref{t:tt-torsion-is-1+m-torsors}.1, this bi-extension becomes trivial upon pullback to $\wt A$. Using the short exact sequence
of $\wt A$,
 we thus obtain a pushout diagram of adic groups relatively over  $A^{\vee\tt}$

	\begin{center}
	\begin{tikzcd}
		0 \arrow[r] &TA\times A^{\vee\tt} \arrow[r] \arrow[d, "\mathbf e"'] & \wt A\times  A^{\vee\tt} \arrow[d] \arrow[r] & A\times  A^{\vee\tt} \arrow[d, equal] \arrow[r] & 0 \\
		0 \arrow[r] &\G^\tt_m \arrow[r] & \mathcal P \arrow[r] & A\times  A^{\vee\tt}\arrow[r] & 0
	\end{tikzcd}
\end{center}
where the middle arrow exists because $\mathcal P$ becomes split on $\wt A\times  A^{\vee\tt}$, and is unique because any two morphisms differ by a map from $\wt A$ to $\G_m^\tt$, which has to be trivial by \cref{p:cohomology-of-wtXxY}.

On the left, we thus obtain the desired bilinear pairing
$\mathbf{e}:TA\times A^{\vee\tt}\to \G^\tt_m$.
It is clear from the construction via pullback to $\wt A$ that this is precisely the pairing from Theorem~\ref{t:geometric-Simpson}.
On the other hand, it is also clear from comparing to the pullback of \eqref{eq:Poincare-bundle-vs-finite-Weil-pairing} to $\wt A$ that this gives an analytic continuation of the adelic Weil pairing
\[ (\mathbf{e}_N)_{N\in\N}:TA\times \textstyle\varinjlim_{N\in\N}A^\vee[N]\to \mu.\]
It remains to see that this property determines $\mathbf{e}$ uniquely: For this we use that by Proposition~\ref{l:hat-vs-p-divisible-groups}, for each $x\in TA$ the morphism $\mathbf{e}(x,-)$ is uniquely determined by its value on prime-to-$p$ torsion and a morphism on the analytic $p$-divisible group $ A\tpt\to\G_m\tpt$. By fully faithfulness of Fargues' functor from $p$-divisible groups to analytic $p$-divisible groups \cite[Théorème~6.1]{Fargues-groupes-analytiques}, this is uniquely determined on $p$-torsion points, as desired.
\end{proof}

\begin{Remark}\label{s:tt-Pic-intro}
As an application of \cref{t:geometric-Simpson}, we point out that the geometric Simpson correspondence can in turn be used to study classical rigid analytic Picard functors: 

It is expected that for any smooth proper rigid space $X$, the Picard functor $\uP_{X,\et}$ is represented by a rigid group, and many instances of this are known, see \cite[\S1]{heuer-diamantine-Picard}. In all known cases, one additionally has the structural result that the  identity component $\uP_{X,\et}^\circ$  is a semi-abeloid variety, i.e.\ an extension of an abeloid  by a rigid torus, see especially \cite{HartlLutk}. But  it  currently seems an open question if one should expect such a description in general.

From this perspective, \cref{t:geometric-Simpson} gives the new result that a topological torsion variant of the Picard functor is always representable.
In particular, if $\uP_{X,\et}$ is representable by a rigid group, this describes its open topological torsion subgroup. From this we can deduce structural results on $\uP_{X,\et}^{\tt}$ which provide evidence that $\uP_{X,\et}^\circ$ might always be representable by a semi-abeloid variety. Namely, \cref{t:geometric-Simpson} imposes concrete structural restrictions on what rigid groups can appear as Picard varieties, like existence of an analytic Weil pairing on topological torsion, which are consistent with $\uP_{X,\et}^{\circ}$ being semi-abeloid.

As a basic example, we recover \cite[Lemma 3.1]{HartlLutk}, which says
$\Hom(\G_a,\uP_{X,\et})=0$.
Indeed, as $\G_a=\G_a^\tt$, any such map factors through $\uP^\tt_{X,\et}\hookrightarrow \Hom(\pi_1(X),\G_m)$, but $\Hom(\G_a,\G_m)=0$. Similarly, if $\uP_{X,\et}$ was an open unit ball with additive structure, this would contradict it being the kernel of a map from $ \Hom(\pi_1(X),\G_m)$ to an affine group.
\end{Remark}
\section{Comparison of analytic moduli spaces}
In this final section, we use our main theorems to show that the $p$-adic Simpson correspondence in rank one admits a geometric description in terms of a comparison of moduli spaces. This also proves \cref{t:intro-Corollary-Cp} and explains the necessary choices in a geometric way.

As before, let $X$ be a connected smooth proper rigid space $X$ over an algebraically closed complete extension $K$ of $\Q_p$ and choose a base-point $x\in X(K)$ to define $\pi_1(X):=\pi_1^{\et}(X,x)$. Let us for simplicity assume in this section that the classical rigid analytic Picard functor is representable (see \cref{s:tt-Pic-intro}).
As an application of \cref{t:tt-torsion-is-1+m-torsors} and \cref{t:geometric-Simpson}, we can now define rigid analytic moduli spaces on both sides of the $p$-adic Simpson correspondence, in very close analogy to Simpson's complex analytic moduli spaces of rank one \cite[\S2]{ModuliRankOne}:
\begin{Definition}
	\begin{enumerate}
		\item The coarse moduli space of $v$-line bundles on $X$ is
		\[\Bun_{v,1}:=\uP_{X,v}\]
		\item   Write $\mathcal A:=H^0(X,\Omega^1(-1))\otimes \G_a$,  this is the Hitchin base of rank one. Then
		\[\Higgs_1:=\uP_{X,\et}\times \mathcal A\]
		is the
		coarse moduli space of Higgs line bundles on $X$.
		\item We define the \textbf{Betti moduli space} to be the character variety of \cref{l:cts-char-variety}, 
		\[\MB:=\HOM(\pi_1(X),\G_m)=\uP_{X,v}^{\tt}\subseteq \uP_v,\]
		 which is the moduli space
		of characters of $\pi_1(X)$. Via \cref{t:tt-torsion-is-1+m-torsors}.2, we can equivalently see this as the coarse moduli space of topologically torsion $v$-line bundles.
		\item We define the \textbf{Dolbeault moduli space} to be the coarse moduli space
		 \[\MD:=\uP_{X,\et}^{\tt}\times \mathcal A \subseteq \uHiggs_1\]
		 of topological torsion Higgs line bundles on $X$.
	\end{enumerate}
\end{Definition}
While a priori defined as $v$-sheaves on $\Perf_K$,
\cref{t:representability-of-v-Picard-functor}, \cref{t:geometric-Simpson} and \cref{l:cts-char-variety} combine to show that all of the above moduli functors are represented by smooth rigid spaces. This allows us to investigate a new perspective on the $p$-adic Simpson correspondence, namely whether it admits a description in terms of moduli spaces.

This turns out to be the case: In this section, we note three slightly different ways in which we can compare $\Bun_{v,1}$ to $\Higgs_1$, respectively $\MB$ to $\MD$. The basic idea for this comparison is that $\HTlog$ on the Betti side, and the projection to $\mathcal A$ on the Dolbeault side, define a diagram
\[ \begin{tikzcd}[row sep = 0cm,column sep = 0.5cm]
	\MB \arrow[r, hook] & {\Bun_{v,1}} \arrow[rd, "\HTlog"] &  \\
	&  & \mathcal A \\
	\MD \arrow[r, hook] & \Higgs_1 \arrow[ru] & 
\end{tikzcd}\]
in which the second column consists of $\uP_{X,\et}$-torsors, and the first column consists of  $\uP^\tt_{X,\et}$-torsors. While the torsors on the bottom line are split,  the torsors on the top line are typically not split according to Corollary~\ref{c:torsors-not-split}. The first, rather primitive way in which the two sides can be compared is now the following ``tautological comparison'', based on the observation that any extension becomes tautologically split after pullback to itself.
\begin{Corollary}\label{c:prim-paCS-comparison}
	There is a canonical and functorial isomorphism of rigid groups
	\[\Bun_{v,1}\times_{\mathcal A}\Bun_{v,1}\isomarrow \Higgs_1\times_{\mathcal A}\Bun_{v,1},\]
	given by sending $(L_1,L_2)\mapsto ((L_1\otimes L_2^{-1},\HTlog(L_1)),L_2)$. It restricts to an isomorphism of rigid groups
	\[ 	\MB \times_{\mathcal A}	\MB \isomarrow 	\MD \times_{\mathcal A}	\MD.\]
\end{Corollary}

Our second comparison isomorphism is geometrically more refined but no longer completely canonical: It relies on the observation that the exact sequence \cref{seq:Picv-tt-ses} in \cref{t:geometric-Simpson} admits a reduction of structure groups to the $p$-torsion subsheaf $\uP_{X,\et}[p^\infty]\subseteq \uP_{X,\et}^\tt$:
\begin{Definition}\label{d:L_XX}
	Let $\mathbb X$ be any $B_{\dR}^+/\xi^2$-lift of $X$, which by \cref{r:splitting-of-HT} induces a splitting $s_{\mathbb X}$ of the Hodge--Tate sequence \cref{r:HT-seq-splitting} . Then we define the \'etale sheaf $\mathbb L_{\mathbb X}\to \mathcal A$ as the pullback
	\[ \begin{tikzcd}
		&  & 0 \arrow[r] & {\uP_{X,\et}[p^\infty]} \arrow[d,equal] \arrow[r] & \mathbb L_{\mathbb X} \arrow[d] \arrow[r] & \mathcal A\arrow[d,"s_{\mathbb X}"] \arrow[r] & 0 \\
		&  & 0 \arrow[r] & {\uP_{X,\et}[p^\infty]}  \arrow[r]& {\HOM(\pi_1(X),\wh\G_m)} \arrow[r,"\log"] &{\HOM(\pi_1(X),\G_a)}  \arrow[r]& 0
	\end{tikzcd}\]
where $\wh\G_m:=\G_m\tpt$ is the open disc of radius $1$ around $1$, and the bottom row is obtained by applying $\HOM(\pi_1(X),-)$ to the logarithm sequence $0\to \mu_{p^\infty}\to \wh\G_m\xrightarrow{\log} \G_a\to 0$.
\end{Definition}
The sheaf $\mathbb L_{\mathbb X}$ induces the following ``\'etale comparison isomorphism'':
\begin{Theorem}\label{t:etale-comparison}
	There is a canonical isomorphism
	\[ \Bun_{v,1}\times_{\mathcal A}\mathbb  L_{\mathbb X}\isomarrow \Higgs_{1}\times_{\mathcal A}\mathbb  L_{\mathbb X}.\]
	It restrict to an isomorphism
	\[ \MB\times_{\mathcal A}\mathbb  L_{\mathbb X}\isomarrow \MD\times_{\mathcal A}\mathbb  L_{\mathbb X}.\]
\end{Theorem}
\begin{proof}
	The defining sequence of $\mathbb  L_{\mathbb X}$ is a reduction of structure group of the short exact sequence  from \cref{t:geometric-Simpson}: 
	This follows from combining the above diagram with
	\[ \begin{tikzcd}
	&  & 0 \arrow[r] & {\uP_{X,\et}[p^\infty]} \arrow[r]  \arrow[d]& {\HOM(\pi_1(X),\wh\G_m)} \arrow[r,"\log"] \arrow[d,hook]&{\HOM(\pi_1(X),\G_a)}  \arrow[r]\arrow[d,"\HT"] & 0 \\
	&  & 0 \arrow[r] &  {\uP_{X,\et}^\tt} \arrow[r] & {\HOM(\pi_1(X),\G_m)} \arrow[r,"\HTlog"] & \mathcal A \arrow[r] & 0
\end{tikzcd}\]
and the fact that $\HT\circ s_{\mathbb X}=\id$. Since  $\mathbb  L_{\mathbb X}\to \mathcal A$ is a $\uP_{X,\et}[p^\infty]$ torsor, there is a canonical isomorphism $\mathbb  L_{\mathbb X}\times_{\mathcal A}\mathbb  L_{\mathbb X}\isomarrow  {\uP_{X,\et}[p^\infty]} \times_{\mathcal A}\mathbb  L_{\mathbb X}$, like in \cref{c:prim-paCS-comparison}. Forming the pushout of the $\uP_{X,\et}[p^\infty]$-torsor in the first factor  along $\uP_{X,\et}[p^\infty]\to \uP_{X,\et}^\tt$ gives the second isomorphism. Pushing further along $\uP_{X,\et}^\tt\to \uP_{X,\et}$ gives the first by \cref{c:pushout-relation}
\end{proof}

Finally, there is an \'etale version which is completely canonical, and which can be described as a universal version of \cref{t:etale-comparison} over the moduli space of all Hodge--Tate splittings:
\begin{Definition}
	Let $\mathcal S$ be the rigid affine space associated to the finite $K$-vector space \[\Hom_K(H^0(X,\Omega^1_X(-1)),\Hom_{\cts}(\pi_1(X),K)).\]
	Set $\mathcal A_{\mathcal S}:=\mathcal A\times \mathcal S$.  Then there is a tautological morphism
	\[ s:\mathcal A_{\mathcal S}\to \HOM(\pi_1(X),\G_a).\]
	We use it to define a universal \'etale sheaf $ \mathbb L\to \mathcal A_{\mathcal S}$ via the commutative diagram
		\[ \begin{tikzcd}
		&  & 0 \arrow[r] & {\uP_{X,\et}[p^\infty]} \arrow[d,equal] \arrow[r] & \mathbb L \arrow[d] \arrow[r] & \mathcal A_{\mathcal S}\arrow[d,"s"] \arrow[r] & 0 \\
		&  & 0 \arrow[r] & {\uP_{X,\et}[p^\infty]}  \arrow[r]& {\HOM(\pi_1(X),\wh\G_m)} \arrow[r,"\log"] &{\HOM(\pi_1(X),\G_a)}  \arrow[r]& 0.
	\end{tikzcd}\]
	Then for any lift $\mathbb X$ of $X$ corresponding to a point $x\in \mathcal S(K)$, the fibre of $\mathbb L$ over $\mathcal A\times x$ is  $\mathbb L_{\mathbb X}$.
	
	We let $\mathcal P$ be the pushout of $\mathbb L$ along  $\uP_{X,\et}[p^\infty]\to \uP_{X,\et}^\tt$. This is an extension
	\[ 0\to \uP_{X,\et}^\tt\to \mathcal P\to \mathcal A_{\mathcal S}\to 0.\]
	Note that this is completely canonical and does not depend on either lifts or exponentials.
\end{Definition}
We now have the following ``universal comparsion isomorphism'': In a way reminiscent of the Deligne--Hitchin twistor space in complex geometry \cite[\S4]{Simpson_Twistor}, this exhibits $\MD$ as a ``degeneration'' of $\MB$ in the sense that there is an analytic morphism over a base $\mathcal S$ whose fibres are generically isomorphic to $\MB$, but whose fibre over a special point is equal to $\MD$.
(This is not to say that the universal comparison is directly related to $\lambda$-connections over $X$, whose closest known analogue are the $q$-connections of Morrow--Tsuji \cite[\S2]{MorrowTsuji}.)
\begin{Proposition}\label{p:universal-comp-morphism}
	\begin{enumerate}
		\item For any $B_{\dR}^+/\xi^2$-lift $\mathbb X$ of $X$, there is for the fibre $\mathcal P_x$  of  $\mathcal P$ over the associated point $x\in \mathcal S(K)$ a canonical isomorphism
		\[\mathcal P_x\isomarrow \MB.\]
		\item For any morphism 
		$x:H^1_{\et}(X,\Omega^1)\to H^1_{\et}(X,\Q_p)\otimes K=\Hom_{\cts}(\pi_1(X),K)$ factoring through the Hodge-Tate filtration, we instead have a canonical isomorphism 
		\[\mathcal P_x\isomarrow \MD.\] 
		\item
		There is a canonical and functorial isomorphism of rigid spaces
		\[c: \mathcal P\times_{\mathcal A_{\mathcal S}}\mathbb L\isomarrow \MD\times_{\mathcal A}\mathbb L\]
		such that for any $B_{\dR}^+/\xi^2$-lift $\mathbb X$ of $X$, the fibre of  $c$ over the associated point $x\in \mathcal S(K)$ is canonically identified via (1) with  the isomorphism from \cref{t:etale-comparison}.
	\end{enumerate}
The analogous statements hold for $\Bun_1$, $\Higgs_{1}$ instead of $\MB$, $\MD$, respectively.
\end{Proposition}
\begin{proof}
	We have a commutative diagram
	\[ \begin{tikzcd}
		&  & 0 \arrow[r] & {\uP_{X,\et}[p^\infty]} \arrow[d] \arrow[r] & \mathbb L \arrow[d] \arrow[r] & \mathcal A_{\mathcal S}\arrow[d,"\HT\circ s"] \arrow[r] & 0 \\
		&  & 0 \arrow[r] & {\uP^\tt_{X,\et}}  \arrow[r]& {\HOM(\pi_1(X),\G_m)} \arrow[r] &{\mathcal A}  \arrow[r]& 0
	\end{tikzcd}\]
	inducing a natural map $\mathcal P\to \HOM(\pi_1(X),\G_m)$ by the universal property of the pushout. 
	
	Part 1 follows from comparing the definition of $\mathbb L$ to that of \cref{d:L_XX}, since in this case the morphism on the right becomes the identity.
	
	To see part 2, observe that in this case $\HT\circ s$ specialises to $0$. Thus there is a canonical morphism $\mathbb L_x\to  \uP^\tt_{X,\et}$ which induces a canonical splitting of $\mathcal P$. This the same datum as an isomorphism $\mathcal P_x\isomarrow \uP^\tt_{X,\et}\times \mathcal A=\MD$.
	
	Finally, part 3 follows from the tautological splitting
	\[\mathbb L\times_{\mathcal A_{\mathcal S}}\mathbb L\isomarrow \uP_{X,\et}[p^\infty]\times_{\mathcal A_S}\mathbb L\]
	after pushout along $\uP_{X,\et}[p^\infty]\to \uP^\tt_{X,\et}$. The last part of 3 about compatibility with the isomorphism from \cref{t:etale-comparison} over $x\in \mathcal S(K)$ is then clear from  comparing the proofs.
\end{proof}

It remains to note that this explains the choices in \cref{t:intro-Corollary-Cp} in a geometric fashion: 
\begin{proof}[Proof of \cref{t:intro-Corollary-Cp}]
It is clear from the definition that any choice of exponential induces a continuous splitting $s$ of $\mathbb L_{\mathbb X}\to \mathcal A$ on $K$-points, which induces the desired homeomorphism
\[ \Hom_{\cts}(\pi_1(X),K^\times)=\MB(K)\isomarrow\MD(K)=\Pic^\tt_X(K)\times H^0(X,\Omega^1)(-1)\]
by taking $K$-points in \cref{t:etale-comparison} and considering the fibre over $s:\mathcal A(K)\to \mathbb L_{\mathbb X}(K)$.
\end{proof}

\begin{Remark}\label{r:etale-corresp}
	As an aside, \cref{t:etale-comparison} means that the term ``geometric $p$-adic Simpson \textit{correspondence}'' in the title may be interpreted quite literally: It provides a diagram
	\[
	\begin{tikzcd}
		{\MB} 	& {\MB\times_{\mathcal A}\mathbb  L_{\mathbb X}}\cong {\MD\times_{\mathcal A}\mathbb  L_{\mathbb X}} \arrow[r]\arrow[l] &   {\MD}
	\end{tikzcd}\]
	which is indeed an ``\'etale correspondence'' in the sense of algebraic geometry.
\end{Remark}

The results of this section give an indication for what kind of geometric results one can expect also in higher rank. More precisely, we believe that the statements of \cref{t:etale-comparison}, \cref{p:universal-comp-morphism} and \cref{r:etale-corresp} stand a chance to generalise: This suggests that the moduli space of $v$-vector bundles is an \'etale twist of the moduli space of Higgs bundles over the Hitchin base, and that the usual choices necessary for the formulation of the $p$-adic Simpson correspondence can be interpreted as leading to a trivialisation of this twist.

\medskip

We will show in \cite{heuer-sheafified-paCS} that indeed, one can construct analytic moduli spaces of $v$-vector bundles and Higgs bundles in terms of small $v$-stacks, and that both admit natural maps to the Hitchin base. This provides further evidence that the results of this article provide the first instance of a general moduli-theoretic approach to the $p$-adic Simpson correspondence.

\bibliographystyle{alphabbrv}
\bibliography{Pic_v-geometric-Pic.bbl}
\end{document}